\documentclass{MathTemplate}
\usepackage{parskip, setspace, parskip, enumitem, float}

\title{Arborescences of Covering Graphs}
\begin{document}

\begin{abstract}
     An \emph{arborescence} of a directed graph $\Gamma$ is a spanning tree directed toward a particular vertex $v$. The arborescences of a graph rooted at a particular vertex may be encoded as a polynomial $A_v(\Gamma)$ representing the sum of the weights of all such arborescences. The arborescences of a graph and the arborescences of a covering graph $\tilde{\Gamma}$ are closely related. Using \emph{voltage graphs} to construct arbitrary regular covers, we derive a novel explicit formula for the ratio of $A_v(\Gamma)$ to the sum of arborescences in the lift $A_{\tilde{v}}(\tilde{\Gamma})$ in terms of the determinant of Chaiken's voltage Laplacian matrix, a generalization of the Laplacian matrix. Chaiken's results on the relationship between the voltage Laplacian and vector fields on $\Gamma$ are reviewed, and we provide a new proof of Chaiken's results via a deletion-contraction argument.
\end{abstract}
\maketitle
\section{Introduction}

In this paper, we examine the relationship between arborescences of a graph and the arborescences of its covering graph. An \emph{arborescence} rooted at a vertex $v$ in a directed graph $\Gamma$ is a weighted spanning tree of $\Gamma$ that is directed towards $v$.  We define $A_v(\Gamma)$ to be the sum of the weights of all arborescences in $\Gamma$ rooted at $v$.  Using the Matrix Tree Theorem \cite[Theorem~5.6.8]{stanley2}, we can compute $A_v(\Gamma)$ as a minor of the Laplacian matrix of $\Gamma$.

It is natural to ask to what extent the arborescences of a graph $\Gamma$ characterize the arborescences of a covering graph $\tilde{\Gamma}$.  Every arborescence of $\Gamma$ lifts to a partial arborescence of $\tilde{\Gamma}$, and this lift is unique if the root of the arborescence in $\tilde{\Gamma}$ is fixed.  Conversely, every arborescence of $\tilde{\Gamma}$ descends to a subgraph of $\Gamma$ containing an arborescence. These properties lead us to ask whether there is a meaningful relationship between $A_v(\Gamma)$ and $A_{\tilde{v}}(\tilde{\Gamma})$, where $\tilde{v}$ is a lift of $v$. We show that $A_v(\Gamma)$ always divides $A_{\tilde{v}}(\tilde{\Gamma})$, meaning that each arborescence of $\Gamma$ corresponds to a set of arborescences of $\tilde{\Gamma}$. The primary goals of this paper are to derive an explicit formula for the ratio $\frac{A_{\tilde{v}}(\tilde{\Gamma})}{A_{v}(\Gamma)}$ and to examine cases where this ratio is especially computationally nice. 

The ratio $\frac{A_{\tilde{v}}(\tilde{\Gamma})}{A_{v}(\Gamma)}$ first arose in Galashin and Pylyavskyy's study of \emph{$R$-systems} \cite{rsystems}. The $R$-system is a discrete dynamical system on a edge-weighted strongly connected simple directed graph $\Gamma = (V, E, \wt)$ whose state vector $X = (X_v)_{v \in V}$ evolves to its next state $X' = (X_v')_{v \in V}$ according to the following relation:
\begin{align}\label{rsys}
    \sum_{(u, v) \in E} \wt(u, v) \frac{X_v}{X_{u}'} = \sum_{(v, w) \in E} \wt(v,w) \frac{X_w}{X_v'}
\end{align}

In (\ref{rsys}), the vertex $v$ is fixed, and the vertices $u$ and $w$ range over the start and endpoints of the ingoing and outgoing edges of $v$, respectively. This system is homogeneous in both $X$ and $X'$, so we consider solutions in projective space. Galashin and Pylyavskyy determined all solutions $X'$ of this equation as a function of $X$:

\begin{theorem}\cite{rsystems}
    The system given by equation (\ref{rsys}) has solution
    \begin{align*}
        X'_v = \frac{X_v}{A_v(\Gamma)}.
    \end{align*}
    This solution is unique up to scalar multiplication, yielding a unique solution to the $R$-system in $\mathbb{P}^{|V|}$. 
\end{theorem}

However, we can see the value of $X'_v$ in equation (\ref{rsys}) depends only on the neighborhood of the vertex $v$. Thus, in the case of a covering graph $\tilde{\Gamma}$, we may find two solutions to the $R$-system: one by applying the previous theorem directly, and one by treating each vertex of $\tilde{\Gamma}$ locally like a vertex of $\Gamma$, and then applying the theorem. The two respective solutions are
\begin{align*}
    X_{\tilde{v}}' = \frac{X_{\tilde{v}}}{A_{\tilde{v}}(\tilde{\Gamma})} & &
    \text{and}& &
    X_{\tilde{v}}' = \frac{X_{\tilde{v}}}{A_{v}({\Gamma})}.
\end{align*}
Therefore, uniqueness of the solution implies that the vectors 
\begin{align*}
    \left(\frac{X_{\tilde{v}}}{A_{\tilde{v}}(\tilde{\Gamma})}\right)_{\tilde{v} \in \tilde{V}} & &
    \text{and} & &
    \left(\frac{X_{\tilde{v}}}{A_{v}({\Gamma})}\right)_{\tilde{v} \in \tilde{V}}
\end{align*}
are scalar multiples of each other, where $\tilde{v}$ is any lift of $v$. Equivalently: 
\begin{corollary}\label{Pasha}
When $\Gamma$ is strongly connected and simple, the ratio $\frac{A_{\tilde{v}}(\tilde{\Gamma})}{A_v(\Gamma)}$ is independent of the choice of vertex $v$ and of the choice of lift $\tilde{v}$.
\end{corollary} 

The existence of this invariance motivates finding an explicit formula for this ratio. 
The following is the main theorem of this paper.
\begin{theorem}\label{bigboi}
	Let $\Gamma = (V, E, \wt)$ be an edge-weighted  multigraph, and let $\tilde{\Gamma}$ be a $k$-fold cover of $\Gamma$ such that each lifted edge is assigned the same weight as its base edge. Denote by $\mathscr{L}(\Gamma)$ the voltage Laplacian of $\Gamma$. Then for any vertex $v$ of $\Gamma$ and any lift $\tilde{v}$ of $v$ in  $\tilde \Gamma$ of $\Gamma$, we have
	\begin{align*}
		\frac{A_{\tilde{v}}(\tilde{\Gamma})}{A_v(\Gamma)} = \frac{1}{k}{\det  [\mathscr{L}(\Gamma)}]_{\BZ[E]}.
	\end{align*}
\end{theorem}

If $\tilde{\Gamma}$ is a regular cover, it is a derived cover by a group $G$ with $|G|=k$. In this case, in the above formula $\det  [\mathscr{L}(\Gamma)]_{\BZ[E]}$ is the determinant of $\mathscr{L}(\Gamma)$ as a $\BZ[E]$-linear transformation. We may evaluate this determinant by restriction of scalars (see Section \ref{sec: main-thm} for details). For an arbitrary cover (including non-regular ones), the matrix $ [\mathscr{L}(\Gamma)]_{\BZ[E]}$ can be determined concretely from the covering graph (Definition \ref{lowerright}), so the determinant can be explicitly computed.

When $\tilde{\Gamma}$ is a regular cover of prime order, we will be able to prove the following refinement in Section \ref{theprimecycliccase}:

\begin{corollary}\label{primecyclic}
	 Let $p$ be a prime, let $\Gamma = (V, E, \wt, \nu)$ be an edge-weighted $\BZ/p\BZ$-voltage directed multigraph, and let $\mathscr{L}(\Gamma)$ be its voltage Laplacian matrix. Then for any vertex $v$ of $\Gamma$ and any lift $\tilde{v}$ of $v$ in the derived graph $\tilde{\Gamma}$ of $\Gamma$, we have
	\begin{align*}
	\frac{A_{\tilde{v}}(\tilde{\Gamma})}{A_v(\Gamma)} &= \frac{1}{|G|}N_{\BQ(\zeta_p)/\BQ}\left(\det \left[\mathscr{L}(\Gamma)\right]\right)\\
	&= \frac{1}{|G|} \prod_{i = 1}^{p -1} \det [\sigma_i (\mathscr{L}(\Gamma))]
	\end{align*}
	where $N_{\BQ(\zeta_p)/\BQ}\left(\det \left[\mathscr{L}(\Gamma)\right]\right)$ denotes the field norm of $\BQ(\zeta_p)$ over $\BQ$, naturally extended to a norm on $\BQ(\zeta_p)[E]$, and $\sigma_i$ is the field automorphism on $\BQ(\zeta_p)$ mapping $\zeta_p \mapsto \zeta_p^i$. 
\end{corollary}

In the case $|G|=2$, we obtain a conjecture by Galashin and Pylyavskyy:

\begin{corollary}\label{2foldtheorem}
    Let $\Gamma = (V, E, \wt, \nu)$ be an edge-weighted $\BZ/2\BZ$-voltage directed multigraph, and let $\mathscr{L}(\Gamma)$ be its voltage Laplacian matrix. Then for any vertex $v$ of $\Gamma$ and any lift $\tilde{v}$ of $v$ in the derived graph $\tilde{\Gamma}$ of $\Gamma$, we have
    $$
        	\frac{A_{\tilde{v}}(\tilde{\Gamma})}{A_v(\Gamma)} = \frac{1}{2}\det \left[\mathscr{L}(\Gamma)\right].
    $$
\end{corollary}

Corollary \ref{2foldtheorem} follows directly from Corollary \ref{primecyclic} by setting $p=2$ and noting that $\sigma_1$ is the identity.

Theorem \ref{bigboi} allows us to easily conclude nice properties about the ratio:

\begin{corollary} \label{integrality}
If the edge weights of $\Gamma$ are indeterminates then the ratio $\frac{A_{\tilde{v}}(\tilde{\Gamma})}{A_v(\Gamma)}$ is a homogeneous polynomial in the edge weights with integer coefficients.
\end{corollary}

\begin{proof}
Since ${\det  [\mathscr{L}(\Gamma)}]_{\BZ[E]} \in \BZ[E]$, Theorem \ref{bigboi} tells us that $\frac{A_{\tilde{v}}(\tilde{\Gamma})}{A_v(\Gamma)} \in \BQ[E]$. Note that every coefficient of $A_v(\Gamma)$ is 1 since every edge of $\Gamma$ has a different weight. Therefore, $A_v(\Gamma)$ is a primitive polynomial over the integers, so by Gauss' lemma, $\frac{A_{\tilde{v}}(\tilde{\Gamma})}{A_v(\Gamma)}\in\BZ[E]$.

Homogeneity follows from the fact that every arborescence of a given graph has the same number of edges.
\end{proof}

We furthermore conjecture the following (see Section \ref{sec: future}):
\begin{conjecture}\label{posconjecture}
    Let $\Gamma$ be a directed graph, $\tilde{\Gamma}$ a $k$-cover of $\Gamma$, $v$ a vertex of $\Gamma$ and $\tilde{v}$ a lift of $v$ in $\tilde{\Gamma}$.  If the edge weights of $\Gamma$ are indeterminates then the polynomial $\frac{A_{\tilde{v}}(\tilde{\Gamma})}{A_v(\Gamma)}$ has positive coefficients.
\end{conjecture}

This conjecture suggests there may be a combinatorial interpretation of $\det[\mathscr{L}(\Gamma)]_{\mathbb{Z}[E]}$.

Conjecture~\ref{posconjecture} is motivated by the positivity property of cluster algebras.  Cluster algebras, introduced by Fomin and Zelevinsky~\cite{CA1}, are rings with a distinguished set of generators called \emph{cluster variables}.  Algebraically independent subsets of cluster variables are called \emph{clusters}.  Galashin and Pylyavskyy's $R$-systems are intimately related to cluster algebras (see Section 9 of~\cite{rsystems}), so it is reasonable to expect that the ratio $\frac{A_{\tilde{v}}(\tilde{\Gamma})}{A_v(\Gamma)}$ that arises in the context of $R$-systems may also be related to cluster algebras.

It is obvious from the definition of a cluster algebra that given a cluster, we can write all cluster variables as a rational function with all positive coefficients in the variables of that cluster.  However, it is also true (and not at all obvious) that this rational function can always be simplified to a Laurent polynomial with positive coefficients.  Our situation is similar: it is obvious that both $A_{\tilde{v}}(\tilde{\Gamma})$ and $A_v(\Gamma)$ have positive coefficients.  By Theorem~\ref{bigboi}, we know the ratio $\frac{A_{\tilde{v}}(\tilde{\Gamma})}{A_v(\Gamma)}$ is a polynomial and we believe that, as in the case of cluster algebras, this polynomial has positive coefficients.

It is also useful to explore the consequences of our main result in a graph theory context, as it describes a relationship between a graph and its covers.

One might be interested in the ratio of the \emph{numbers} of arborescences of $\tilde{\Gamma}$ and of $\Gamma$. We can get this number from Theorem \ref{bigboi} by specializing the edge weights to 1:

\begin{align*}
    \frac{A_{\tilde{v}}(\tilde{\Gamma})|_{wt=1}}{A_v(\Gamma)|_{wt=1}} = \left.\frac{1}{k}{\det  [\mathscr{L}(\Gamma)}]_{\BZ[E]}\right|_{wt=1}.
\end{align*}

It is conceivable that this special case might be easier to prove than the general result, but we have not found this to be the case. It is worth noting that the positivity conjecture is trivial in this setting, since the number of arborescences must always be nonnegative.

Our main theorem is an example of a result relating an invariant of a graph to the same invariant on its cover. It would be interesting to try to describe similar results on a larger class of graph invariants.

\begin{problem} \label{invariant-problem}
    For which graph invariants $I$ does the ratio $\frac{I(\tilde{\Gamma})}{I(\Gamma)}$ have nice properties? For example, when can the ratio be expressed as an integer or a polynomial with integer coefficients? When does the ratio have positivity properties?
\end{problem}

One invariant for which our work gives a partial answer is the number of Euler circuits. An Euler circuit in a directed graph is a cycle that uses each edge of the graph precisely once. Euler showed that a graph has an Euler circuit precisely when the graph is connected and every vertex has the same number of edges in and out. In this case, the so-called ``BEST'' Theorem \cite{vanAardenneEhrenfest} gives a formula for the number of Euler circuits $E(G)$ of a graph $G$:

\[E(G) = A_v(G)|_{\wt=1} \cdot \prod_{v \text{ vertex in } G} (\deg v - 1)!,\] where $\deg v$ is the outdegree of $v$.

We can combine this formula with Theorem \ref{bigboi} to say:

\begin{corollary}
\[\frac{E(\tilde{\Gamma})}{E(\Gamma)} = \frac{1}{k}\left.\det  [\mathscr{L}(\Gamma)]_{\BZ[E]}\right|_{\wt=1} \cdot \prod_v ((\deg v - 1)!)^{k-1},\] where the product is over the vertices in $\Gamma$. This quantity is a positive integer.
\end{corollary}

A recent example of this approach was taken by Verma \cite{verma}, who explored the ratio of Tutte polynomials between a graph and its cover. For certain very special graphs and a particular specialization of the Tutte polynomial, the ratio in Problem \ref{invariant-problem} is nice, although in general, the ratio of Tutte polynomials does behave nearly as nicely as the ratio of arborescence sums.

The rest of the paper will proceed as follows.  Section \ref{backgroundsection} covers the background and conventions necessary to read this paper. In this section, we also discuss the Laplacian matrix and the Matrix Tree Theorem in greater detail, and give additional topological background on covering graphs. In particular, we introduce the \emph{voltage graph}, a construction that allows us to compactly describe arbitrary regular covering graphs $\tilde{\Gamma}$ by assigning a group-valued voltage to each edge of $\Gamma$. In Section~\ref{sec: main-thm}, we prove the main theorem.  We also describe restriction of scalars and prove Corollary \ref{primecyclic}.  Section \ref{sec: vec-fields} reviews some known results relating \emph{vector fields} on voltage graphs to the voltage Laplacian. Vector fields are closely related to arborescences, and this discussion especially helps to frame the results of the case of $2$-fold covers. We conclude with several open questions in Section~\ref{sec: future}.

\section{Background and Definitions} \label{backgroundsection}

\subsection{Arborescences}
Let $\Gamma = (V, E, \wt)$ be an edge-weighted directed multigraph with a weight function on the edges $\wt: E \to R$, for some ring $R$. We will usually abbreviate ``edge-weighted" to ``weighted" and ``directed multigraph" to ``graph." We will consider the weights of the edges of $G$ to be indeterminates, treating the weight $\wt(e)$ of an edge $e$ as a variable.  Let the set of such variables be denoted $\wt(E)$.  We denote the source vertex of an edge $e$ by $e_s$ and target vertex of $e$ by $e_t$. If an edge has source $v$ and target $w$, we may write $e = (v, w)$. However, note that when $\Gamma$ is not necessarily simple, there may be more than one edge satisfying these properties, so $(v,w)$ may specify multiple edges.  We denote the set of outgoing edges of a vertex $v$ by $E_s(v)$, and the set of incoming edges of $v$ by $E_t(v)$.

\begin{definition}An \emph{arborescence} $T$ of $\Gamma$ rooted at $v \in V$ is a spanning tree directed towards $v$. That is, for all vertices $w$, there exists a unique directed path from $w$ to $v$ in $T$. \footnote{In the literature, an arborescence rooted at $v$ is usually defined to be a spanning tree directed \emph{away} from $v$, so that $v$ is the unique source rather than the unique sink; see, for example, \cite{korte}, \cite{chaiken}, and \cite{gordon}. Our convention is consistent with the study of $R$-systems.} We denote the set of arborescences of $\Gamma$ rooted at vertex $v$ by $\mathcal{T}_v(\Gamma)$. The \emph{weight of an arborescence} $\wt(T)$ is the product of the weights of its edges:
\begin{align*}
    \wt(T) = \prod_{e \in T} \wt(e)
\end{align*}

We denote by $A_v(\Gamma)$ the sum of the weights of all arborescences of $\Gamma$ rooted at $v$:
\begin{align*}
    A_v(\Gamma) = \sum_{T \in \mathcal{T}_v(\Gamma)} \wt(T)
\end{align*}
\end{definition}
$A_v(\Gamma)$ is either zero or a homogeneous polynomial of degree $|V| - 1$ in the edge weights of $G$. 

\begin{example}\label{ex:arb}
Consider the following edge-weighted directed graph $\Gamma$.
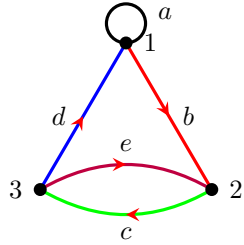
\begin{center}
\begin{tikzpicture}
   [scale=0.65,line width=1.2]
    \coordinate (1) at (0, 3);
    \coordinate (2) at (3/1.71, 0);
    \coordinate (3) at (-3/1.71, 0);
    
    \coordinate (4) at (0, 3 + 0.4);
    \draw (1) arc(270:360+270:0.4);

    \path [draw = red, postaction = {on each segment = {mid arrow = red}}]
    (1)--(2);
     \path [draw = blue, postaction = {on each segment = {mid arrow = red}}]
     (3)--(1);
     \path [draw = green, postaction = {on each segment = {mid arrow = red}}]
    (2) to [bend left] (3);
    
     \path [draw = purple, postaction = {on each segment = {mid arrow = red}}]
    (3) to [bend left] (2);
    
    \draw[fill] (1) circle [radius=0.1];
    \node at (0.5, 3) {$1$};
    \draw[fill] (2) circle [radius=0.1];
    \node at (3/1.71 + 0.5, 0) {2};
    \draw[fill] (3) circle [radius=0.1];
    \node at (-3/1.71-0.5, 0) {3};

    \node at (0.8, 3+ 0.6) {$a$};
    \node at (3/1.71/2 + 0.4, 3/2) {$b$};
    \node at (-3/1.71/2 - 0.5, 3/2) {$d$};
    \node at (0, 0.85 + 0.05) {$e$};
    \node at (0, -0.85 - 0.05) {$c$};
    \end{tikzpicture}
\end{center}
The arborescences rooted at 2 are:
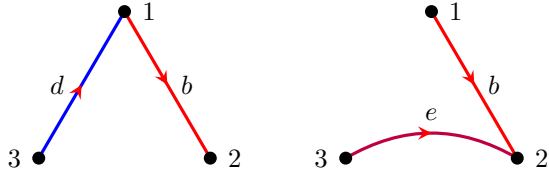
\begin{center}
\begin{tikzpicture}
   [scale=0.65,line width=1.2]
    \coordinate (1) at (0, 3);
    \coordinate (2) at (3/1.71, 0);
    \coordinate (3) at (-3/1.71, 0);
    
    \coordinate (4) at (0, 3 + 0.4);

    \path [draw = red, postaction = {on each segment = {mid arrow = red}}]
    (1)--(2);
    \path [draw = blue, postaction = {on each segment = {mid arrow = red}}]
    (3)--(1);
    
    \draw[fill] (1) circle [radius=0.1];
    \node at (0.5, 3) {$1$};
    \draw[fill] (2) circle [radius=0.1];
    \node at (3/1.71 + 0.5, 0) {2};
    \draw[fill] (3) circle [radius=0.1];
    \node at (-3/1.71-0.5, 0) {3};

    \node at (3/1.71/2 + 0.4, 3/2) {$b$};
    \node at (-3/1.71/2 - 0.5, 3/2) {$d$};
    \end{tikzpicture}
    \hspace{0.2in}
    \begin{tikzpicture}
   [scale=0.65,line width=1.2]
    \coordinate (1) at (0, 3);
    \coordinate (2) at (3/1.71, 0);
    \coordinate (3) at (-3/1.71, 0);
    
    \coordinate (4) at (0, 3 + 0.4);

    \path [draw = red, postaction = {on each segment = {mid arrow = red}}]
    (1)--(2);
    
    \path [draw = purple, postaction = {on each segment = {mid arrow = red}}]
    (3) to [bend left] (2);
    
    \draw[fill] (1) circle [radius=0.1];
    \node at (0.5, 3) {$1$};
    \draw[fill] (2) circle [radius=0.1];
    \node at (3/1.71 + 0.5, 0) {2};
    \draw[fill] (3) circle [radius=0.1];
    \node at (-3/1.71-0.5, 0) {3};

    \node at (3/1.71/2 + 0.4, 3/2) {$b$};
    \node at (0, 0.85 + 0.05) {$e$};
    \end{tikzpicture}
\end{center}
The weight of the arborescence on the left is $bd$ and the weight of the arborescence on the right is $be$.  So, $A_2(\Gamma)=bd+be$.
\end{example}

\subsection{The Laplacian matrix and the Matrix Tree Theorem}

The Matrix Tree Theorem, also known as Kirchoff's Theorem, yields a way of computing $A_v(\Gamma)$ through the \emph{Laplacian matrix} of $\Gamma$. 

\begin{definition} \label{Laplacian}

Label the vertices of $\Gamma$ as $v_1, v_2, \dots$. The \emph{Laplacian matrix} $L(\Gamma)$ of a graph $\Gamma$ is the difference of the weighted degree matrix $D$ and the weighted adjacency matrix $A$ of $\Gamma$:
\begin{align*}
    L(\Gamma) = D(\Gamma) - A(\Gamma).
\end{align*}
Here, the weighted degree matrix is the diagonal matrix whose $i$-th entry is 
\begin{align*}
    d_{ii} = \sum_{e \in E_{s}(v_i)} \wt(e)
\end{align*} and the weighted adjacency matrix has entries defined by 
\begin{align*}
    a_{ij} = \sum_{e = (v_i, v_j)} \wt(e).
\end{align*}
\end{definition}

Since we will always be working with weighted graphs in this paper, we will usually drop the word ``weighted" when talking about the Laplacian matrix. Note also the ordering of the rows and columns of the Laplacian. We will always assume that $v_1$ corresponds to the first row and column of $L(\Gamma)$, that $v_2$ corresponds to the second row and column of $L(\Gamma)$, and so on.

\begin{theorem}\label{MTT2}
\emph{(Matrix Tree Theorem)} \cite{chaiken} The sum of the weights of arborescences rooted at $v_i$ is equal to the minor of $L(\Gamma)$ obtained by removing the $i$-th row and column:
\begin{align*}
    A_{v_i}(\Gamma) = \det [L_{i}^i(\Gamma)].
\end{align*}
\end{theorem}

\begin{example} \label{ex:laplacian}
For the graph $\Gamma$ from Example~\ref{ex:arb}, we have $$D(\Gamma)=\begin{bmatrix}
a+b & 0 & 0 \\
0 & c & 0 \\
0 & 0 & d+e
\end{bmatrix},
\qquad
A(\Gamma)=\begin{bmatrix}
a & b & 0 \\
0 & 0 & c \\
d & e & 0
\end{bmatrix}, \qquad
L(\Gamma)=\begin{bmatrix}
b & -b & 0 \\
0 & c & -c \\
-d & -e & d+e
\end{bmatrix}.$$
Removing row and column 2 and taking the determinant, we get $$\left|\begin{matrix}
b & 0 \\
-d & d+e
\end{matrix}\right|=bd+be.$$  As we computed in Example~\ref{ex:arb}, this is $A_2(\Gamma)$.
\end{example}

\subsection{Covering graphs, voltage graphs, and derived graphs}

\begin{definition} A \emph{k-fold cover} of $\Gamma = (V, E)$ is a graph $\tilde{\Gamma} = (\tilde{V}, \tilde{E})$ that is a $k$-fold covering space of $G$ in the topological sense that preserves edge weight. That is, we require that a lifted edge in the covering graph has the same weight as its corresponding base edge in the base graph. In order to use this definition, we need to find a way to formally topologize directed graphs in a way that encodes edge orientation. To avoid this, we instead give a more concrete alternative definition of a covering graph. The graph $\tilde{\Gamma} = (\tilde{V}, \tilde{E})$ is a $k$-fold covering graph of $\Gamma = (V, E)$ if there exists a projection map $\pi: \tilde{\Gamma} \rightarrow \Gamma$ such that
\begin{enumerate}[nosep]
    \item $\pi$ maps vertices to vertices and edges to edges;
    \item $|\pi^{-1}(v)| = |\pi^{-1}(e)| = k$ for all $v \in V, e \in E$;
    \item For all $\tilde{e} \in \tilde{E}$, we have $\wt(\tilde{e}) = \wt(\pi(\tilde{e}))$;
    \item $\pi$ is a local homeomorphism.  Equivalently, $|E_s(\tilde{v})| = |E_s(\pi(\tilde{v}))|$ and $|E_t(\tilde{v})| = |E_t(\pi(\tilde{v}))|$ for all $\tilde{v} \in \tilde{V}$.
\end{enumerate}
When we refer to $\tilde{\Gamma}$ as a covering graph of $\Gamma$, we assume a single distinguished projection $\pi: \tilde{\Gamma} \rightarrow \Gamma$ has been fixed. 
\end{definition}

 We do not require a covering graph to be connected. However, disconnected graphs contain no arborescences, so our main quantity of interest $A_v(\tilde{\Gamma})$ is always $0$ in the disconnected case. 

\begin{definition} A \emph{weighted permutation-voltage graph} $\Gamma = (V, E, \wt, \nu)$ is a weighted directed multigraph with each edge $e$ also labeled by a permutation $\nu(e) = \sigma_e\in S_k$, the symmetric group on $k$ letters. This labeling is called the \emph{voltage} of the edge $e$. Note that the voltage of an edge $e$ is entirely distinct from the weight of $e$.
\end{definition}

\begin{definition} Given a permutation-voltage graph $\Gamma$, we may construct an $k$-fold covering graph $\tilde{\Gamma} = (\tilde{V}, \tilde{E}, \wt)$ of $\Gamma$. $\tilde{\Gamma}$ is a graph with vertex set $\tilde{V} = V \times \{1,2,\ldots,k\}$ and edge set
\begin{align*}
    \tilde{E}:=\left\{\left[v \times x, w \times \sigma_e(x)\right] : x \in \{1,\ldots,k\}, e=(v,w)\in\Gamma\right\}.
\end{align*}
\end{definition}

Every covering graph of $\Gamma$ can be constructed in this way.

\begin{example}\label{ex:perm}
Let $\Gamma$ be the permutation-voltage graph shown in Figure~\ref{fig:perm-voltage-graph}, where edges labeled $(x, y)$ have edge weight $x$ and voltage $y$.  Then we can construct a $k$-fold cover $\tilde{\Gamma}$, with vertices $(v, y) = v^y$ and with edges labeled by weight, is shown in Figure~\ref{fig:perm-derived-graph}.
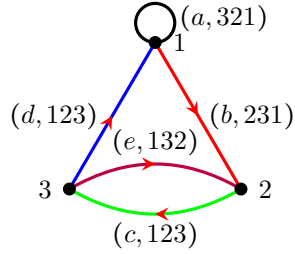
\begin{figure}[htp]\label{fig:derived-graph}
    \centering
    \begin{tikzpicture}
   [scale=0.65,line width=1.2]
    \coordinate (1) at (0, 3);
    \coordinate (2) at (3/1.71, 0);
    \coordinate (3) at (-3/1.71, 0);
    
    \coordinate (4) at (0, 3 + 0.4);
    \draw (1) arc(270:360+270:0.4);

    \path [draw = red, postaction = {on each segment = {mid arrow = red}}]
    (1)--(2);
     \path [draw = blue, postaction = {on each segment = {mid arrow = red}}]
     (3)--(1);
     \path [draw = green, postaction = {on each segment = {mid arrow = red}}]
    (2) to [bend left] (3);
    
     \path [draw = purple, postaction = {on each segment = {mid arrow = red}}]
    (3) to [bend left] (2);
    
    \draw[fill] (1) circle [radius=0.1];
    \node at (0.5, 3) {$1$};
    \draw[fill] (2) circle [radius=0.1];
    \node at (3/1.71 + 0.5, 0) {2};
    \draw[fill] (3) circle [radius=0.1];
    \node at (-3/1.71-0.5, 0) {3};

    \node at (1 + 0.4, 3+ 0.5) {$(a, 321)$};
    \node at (3/1.71/2 + 1.1, 3/2) {$(b, 231)$};
    \node at (-3/1.71/2 - 1.2, 3/2) {$(d, 123)$};
    \node at (0, 0.85 + 0.1) {$(e, 132)$};
    \node at (0, -0.85 - 0.1) {$(c, 123)$};
    \end{tikzpicture}
    \caption{A permutation-voltage graph $\Gamma$.}
    \label{fig:perm-voltage-graph}
\end{figure}
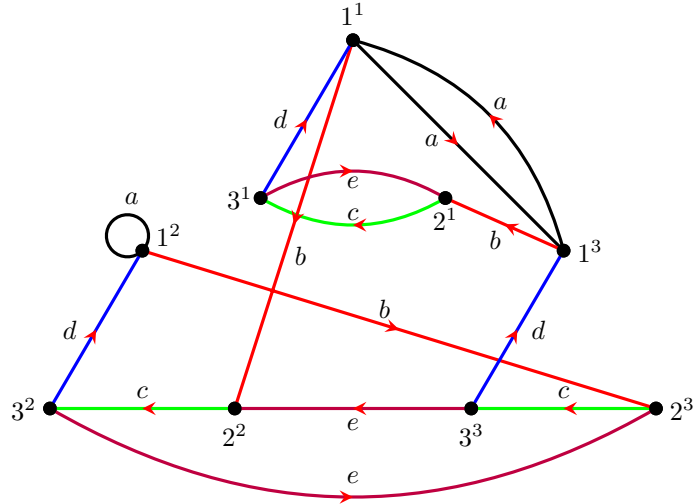
\begin{figure}[htp]
    \centering
  \begin{tikzpicture}[scale = 0.7, line width=1.2]
    \coordinate (1) at (0, 3+4);
    
    \coordinate (2) at (3/1.71, 4);
    \coordinate (3) at (-3/1.71, 4);
    
    \coordinate (4) at (-4, 3);
    
    \coordinate (5) at (3/1.71 - 4, 0);
    \coordinate (6) at (-3/1.71 - 4, 0);
    
    \coordinate (7) at (4, 3);
    
    \coordinate (8) at (3/1.71 + 4, 0);
    \coordinate (9) at (-3/1.71 + 4, 0);
    
\draw[fill] (1) circle [radius=0.1] ;
    \draw[fill] (2) circle [radius=0.1];
    \draw[fill] (3) circle [radius=0.1];
    \draw[fill] (4) circle [radius=0.1];
    \draw[fill] (5) circle [radius=0.1];
    \draw[fill] (6) circle [radius=0.1];
    \draw[fill] (7) circle [radius=0.1];
    \draw[fill] (8) circle [radius=0.1];
    \draw[fill] (9) circle [radius=0.1];
    
    \node at (0, 7.5) {$1^1$};
    
    \node at (3/1.71, 3.6) {$2^1$};
    \node at (-3/1.71 -0.4, 4) {$3^1$};
    
    \node at (-4 + 0.5, 3+0.3) {$1^2$};
    
    \node at (3/1.71 - 4, -0.5) {$2^2$};
    \node at (-3/1.71 - 4 -0.5, 0) {$3^2$};
    
    \node at (4 + 0.5, 3) {$1^3$};
    
    \node at (3/1.71 + 4 + 0.5, 0) {$2^3$};
    \node at (-3/1.71 + 4, -0.5) {$3^3$};
    
    \path[draw = black, postaction = {on each segment = {mid arrow = red}}]
    
    (1) to (7)
    (7) to [bend right] (1);
    
    \draw (4) arc(-45:360-45:0.4);
    
    \path[draw = green, postaction = {on each segment = {mid arrow = red}}]
    (5) -- (6)
    (8) to (9)
    (2) to [bend left] (3);
    
    \path[draw = red, postaction = {on each segment = {mid arrow = red}}]
    (1)  -- (5)
    (4) -- (8)
    (7) -- (2);
    
    \path[draw = blue, postaction = {on each segment = {mid arrow = red}}]
    (3) -- (1)
    (6) -- (4)
    (9) -- (7);
    
    \path[draw = purple, postaction = {on each segment = {mid arrow = red}}]
    (3) to [bend left] (2)
    (6) to [bend right] (8)
    (9) -- (5);
    
    \node at (1.5, 5.1) {$a$};
    \node at (2.8, 5.8) {$a$};
    \node at (-4.2, 4) {$a$};
    
    \node at (0.6, 1.9) {$b$};
    \node at (-1, 2.9) {$b$};
    \node at (2.7, 3.2) {$b$};
    
    \node at (0, 4-0.3) {$c$};
    \node at (-4, 0.3) {$c$};\node at (4, 0.3) {$c$};
    
    \node at (-3/1.71/2 - 0.5, 3/2+4) {$d$};
    \node at (-3/1.71/2 - 0.5-4, 3/2) {$d$};
    \node at (-3/1.71/2 + 0.4+4, 3/2) {$d$};
    
    \node at (0, 4+0.3) {$e$};
    \node at (0, -1.3) {$e$};
    \node at (0, -0.3) {$e$};
    
    \draw[fill] (1) circle [radius=0.1] ;
    \draw[fill] (2) circle [radius=0.1];
    \draw[fill] (3) circle [radius=0.1];
    \draw[fill] (4) circle [radius=0.1];
    \draw[fill] (5) circle [radius=0.1];
    \draw[fill] (6) circle [radius=0.1];
    \draw[fill] (7) circle [radius=0.1];
    \draw[fill] (8) circle [radius=0.1];
    \draw[fill] (9) circle [radius=0.1];
    \end{tikzpicture}
    \caption{The derived covering graph $\tilde{\Gamma}$ of $\Gamma$ in Figure~\ref{fig:voltage-graph}. Edge colors denote correspondence to the edges of $\Gamma$ via the quotient map.}
\label{fig:perm-derived-graph}
\end{figure}
\end{example}

There is a special case of the above definitions that will be particularly useful for us. Let $G$ be a finite group of size $k$. Instead of assigning each lift of an edge an integer in $\{1,\ldots,k\}$, we assign it an element $g$ of $G$. Each edge $e$ of the base graph is assigned a group element $g_e$, and the permutation-voltage $\sigma_e$ is obtained by the action of the group: $\sigma_e(x) := \nu(e)\cdot x$.  We will abuse notation and write $\nu(e) = g_e$. In this context, we will call $\Gamma$ a $G$-voltage graph.

Given a $G$-voltage graph $\Gamma$, the associated $|G|$-fold covering graph $\tilde{\Gamma}$ of $\Gamma$ is known as the \emph{derived graph}. The vertex set of $\tilde{\Gamma}$ is $\tilde{V} = V \times G$ and edge set
\begin{align*}
    \tilde{E}:=\left\{\left[v \times x, w \times (gx)\right] : x \in G, e=(v,w)\in\Gamma, \nu(e)=g\in G\right\}.
\end{align*}

\begin{example}\label{ex1}
Let $G = \BZ/3\BZ = \{1, g, g^2\}$, and let $\Gamma$ be the $G$-voltage graph shown in Figure~\ref{fig:voltage-graph}, where edges labeled $(x, y)$ have edge weight $x$ and voltage $y$.  Then the derived graph $\tilde{\Gamma}$, with vertices $(v, y) = v^y$ and with edges labeled by weight, is shown in Figure~\ref{fig:derived-graph}.
\begin{figure}[htp]\label{fig:derived-graph}
    \centering
    \begin{tikzpicture}
   [scale=0.65,line width=1.2]
    \coordinate (1) at (0, 3);
    \coordinate (2) at (3/1.71, 0);
    \coordinate (3) at (-3/1.71, 0);
    
    \coordinate (4) at (0, 3 + 0.4);
    \draw (1) arc(270:360+270:0.4);

    \path [draw = red, postaction = {on each segment = {mid arrow = red}}]
    (1)--(2);
     \path [draw = blue, postaction = {on each segment = {mid arrow = red}}]
     (3)--(1);
     \path [draw = green, postaction = {on each segment = {mid arrow = red}}]
    (2) to [bend left] (3);
    
     \path [draw = purple, postaction = {on each segment = {mid arrow = red}}]
    (3) to [bend left] (2);
    
    \draw[fill] (1) circle [radius=0.1];
    \node at (0.5, 3) {$1$};
    \draw[fill] (2) circle [radius=0.1];
    \node at (3/1.71 + 0.5, 0) {2};
    \draw[fill] (3) circle [radius=0.1];
    \node at (-3/1.71-0.5, 0) {3};

    \node at (1 + 0.2, 3+ 0.5) {$(a, g)$};
    \node at (3/1.71/2 + 0.8, 3/2) {$(b, 1)$};
    \node at (-3/1.71/2 - 1, 3/2) {$(d, g^2)$};
    \node at (0, 0.85 + 0.1) {$(e, 1)$};
    \node at (0, -0.85 - 0.1) {$(c, g^2)$};
    \end{tikzpicture}
    \caption{A $\BZ/3\BZ$-voltage graph $\Gamma$.}
    \label{fig:voltage-graph}
\end{figure}
\begin{figure}[htp]
    \centering
  \begin{tikzpicture}[scale = 0.7, line width=1.2]
    \coordinate (1) at (0, 3+4);
    
    \coordinate (2) at (3/1.71, 4);
    \coordinate (3) at (-3/1.71, 4);
    
    \coordinate (4) at (-4, 3);
    
    \coordinate (5) at (3/1.71 - 4, 0);
    \coordinate (6) at (-3/1.71 - 4, 0);
    
    \coordinate (7) at (4, 3);
    
    \coordinate (8) at (3/1.71 + 4, 0);
    \coordinate (9) at (-3/1.71 + 4, 0);
    
\draw[fill] (1) circle [radius=0.1] ;
    \draw[fill] (2) circle [radius=0.1];
    \draw[fill] (3) circle [radius=0.1];
    \draw[fill] (4) circle [radius=0.1];
    \draw[fill] (5) circle [radius=0.1];
    \draw[fill] (6) circle [radius=0.1];
    \draw[fill] (7) circle [radius=0.1];
    \draw[fill] (8) circle [radius=0.1];
    \draw[fill] (9) circle [radius=0.1];
    
    \node at (0 + 0.5, 3+4) {$1^1$};
    
    \node at (3/1.71 + 0.5, 4) {$2^1$};
    \node at (-3/1.71 -0.3, 4) {$3^1$};
    
    \node at (-4 + 0.5, 3+0.3) {$1^g$};
    
    \node at (3/1.71 - 4 + 0.5, 0) {$2^g$};
    \node at (-3/1.71 - 4 -0.5, 0) {$3^g$};
    
    \node at (4 + 0.5, 3) {$1^{g^2}$};
    
    \node at (3/1.71 + 4 + 0.5, 0) {$2^{g^2}$};
    \node at (-3/1.71 + 4 - 0.5, 0) {$3^{g^2}$};
    
    \path[draw = black, postaction = {on each segment = {mid arrow = red}}]
    
     (1) to [bend right] (4)
    (4) -- (7)
    (7) -- (1);
    
    \path[draw = red, postaction = {on each segment = {mid arrow = red}}]
    (1)  -- (2)
    (4) -- (5)
    (7) -- (8);
    
    \path[draw = green, postaction = {on each segment = {mid arrow = red}}]
    (5) -- (3)
    (8) to [bend left] (6)
    (2) -- (9);
    
    \path[draw = blue, postaction = {on each segment = {mid arrow = red}}]
    (3) -- (7)
    (6) -- (1)
    (9) -- (4);
    
    \path[draw = purple, postaction = {on each segment = {mid arrow = red}}]
    (3) -- (2)
    (6) -- (5)
    (9) -- (8);
    
    \node at (-3, 5.8) {$a$};
    \node at (0, 2.6) {$a$};
    \node at (2.6, 5) {$a$};
    
    \node at (0.5, 5.6) {$b$};
    \node at (5.2, 1.7) {$b$};
    \node at (-3.5, 1.6) {$b$};
    
    \node at (1.7, 2) {$c$};
    \node at (0, -1.3) {$c$};
    \node at (-2.3, 1.7) {$c$};
    
    \node at (-3.1, 3.7) {$d$};
    \node at (-.5, 1) {$d$};
    \node at (1.3, 3.7) {$d$};
    
    \node at (0, 4+0.3) {$e$};
    \node at (4, .3) {$e$};
    \node at (-4, .3) {$e$};
    
    \draw[fill] (1) circle [radius=0.1] ;
    \draw[fill] (2) circle [radius=0.1];
    \draw[fill] (3) circle [radius=0.1];
    \draw[fill] (4) circle [radius=0.1];
    \draw[fill] (5) circle [radius=0.1];
    \draw[fill] (6) circle [radius=0.1];
    \draw[fill] (7) circle [radius=0.1];
    \draw[fill] (8) circle [radius=0.1];
    \draw[fill] (9) circle [radius=0.1];
    \end{tikzpicture}
    \caption{The derived covering graph $\tilde{\Gamma}$ of $\Gamma$ in Figure~\ref{fig:voltage-graph}. Edge colors denote correspondence to the edges of $\Gamma$ via the quotient map.}
\label{fig:derived-graph}
\end{figure}
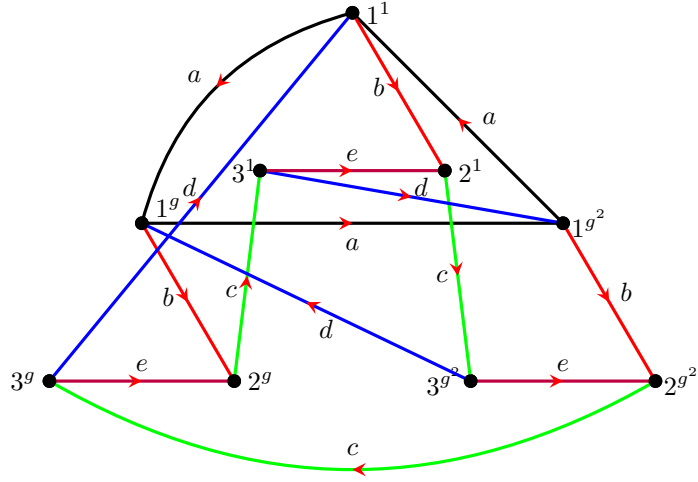
\end{example}

While derived graphs in one sense are a very special subclass of covering graphs, they actually account for all \emph{regular covering graphs}.

\begin{definition}
    Given a graph $\Gamma$ and a covering graph $\tilde{\Gamma}$, the \emph{deck group} $\Aut(\pi)$ of $\tilde{\Gamma}$ is the subgroup of graph automorphisms on $\tilde{\Gamma}$ that commute with $\pi$.
\end{definition}

\begin{definition}
 A \emph{regular cover} $\tilde{\Gamma}$, sometimes known as a \emph{Galois cover}, of a graph $\Gamma$ is a covering graph whose deck group is transitive on each fiber $\pi^{-1}(v)$ for each $v \in V$.  
\end{definition}

\begin{example}
Note that the vertex $1^1$ in Example~\ref{ex:perm} is part of a 2-cycle but vertex $1^2$ is not.  Thus, there is no automorphism that maps $1^1$ to $1^2$.  This means that $\tilde{\Gamma}$ is not a regular cover.

On the other hand, the derived graph in Example~\ref{ex1} is a regular cover because the cyclic permutation $\sigma$ that sends each $v_{i, x}$ to $v_{i, gx}$ is in $\Aut(\pi)$, which shows that $\Aut(\pi)$ is transitive on each fiber $\pi^{-1}(v)$.
\end{example}

\begin{theorem}\emph{(Theorems 3 and 4 in \cite{topograph})}\label{regular}
Every regular cover $\tilde{\Gamma}$ of a graph $\Gamma$ may be realized as a derived cover of $\Gamma$ with voltage assignments in $\Aut(\pi)$. Conversely, every derived graph is a regular cover.
\end{theorem}

The majority of this paper explores the relationship between the arborescences of a voltage graph $\Gamma$ and the arborescences of its derived graph $\tilde{\Gamma}$. Theorem \ref{regular} allows us to deal with all regular covering graphs in the framework of a voltage. It turns out that regularity is not necessary for Theorem \ref{bigboi}, which holds for all $k$-fold covers; however, the results of this main theorem have nice interpretations in terms of the voltage Laplacian in the regular case. 

\subsection{Constructing arborescences of a covering graph: failure of the obvious approach}

In this subsection, we discuss the relationship between arborescences of $\Gamma$ to arborescences of its cover $\tilde{\Gamma}$. If there were a simple correspondence of arborescences of $\Gamma$ with (sets of) arborescences of $\tilde{\Gamma}$, this could lead to a nice combinatorial proof of Theorem \ref{bigboi}. Unfortunately, we have not found such a relationship; we illustrate the pitfalls.

Let $T$ be an arborescence of the base graph $\Gamma$ rooted at $v$. Given a fixed lift $\tilde{v}$ of $v$ in the covering graph $\tilde{G}$, there exists a unique connected lift of $T$ to $\tilde{\Gamma}$ by the local homeomorphism property of covers. The resulting subtree of $\tilde{\Gamma}$ could potentially be completed to a full arborescence of $\tilde{\Gamma}$, possibly in multiple ways, by choosing an outgoing edge of the remaining vertices in such a way as to avoid creating cycles. We project these edges down to $k - 1$ \emph{vector fields} on $\Gamma$, using some method of partitioning the edges into $k - 1$ vector field classes. It is therefore natural to conjecture that arborescences of $\tilde{\Gamma}$ can be put into correspondence with arborescences of $\Gamma$ by utilizing this construction: each arborescence of $\Gamma$ corresponds to a set of arborescences of $\tilde{\Gamma}$ by lifting and then filling in the remaining edges in various combinations according to some nicely enumerable vector field pattern. This conjecture was our main motivation for our study of vector fields; see Section \ref{sec: vec-fields} for more discussion on vector fields. 

Unfortunately, it is not true that every arborescence of $\tilde{\Gamma}$ must stem from such a construction. As a counterexample, consider the base graph $\Gamma$:
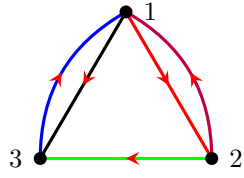
\begin{center}
\begin{tikzpicture}
   [scale=0.65,line width=1.2]
    \coordinate (1) at (0, 3);
    \coordinate (2) at (3/1.71, 0);
    \coordinate (3) at (-3/1.71, 0);
    
    \coordinate (4) at (0, 3 + 0.4);

    \path [draw = red, postaction = {on each segment = {mid arrow = red}}]
    (1)--(2);
    \path [draw = purple, postaction = {on each segment = {mid arrow = red}}]
    (2) to [bend right] (1);
     \path [draw = blue, postaction = {on each segment = {mid arrow = red}}]
     (3) to [bend left] (1);
     \path [draw = black, postaction = {on each segment = {mid arrow = red}}]
     (1) -- (3);
     \path [draw = green, postaction = {on each segment = {mid arrow = red}}]
    (2) -- (3);
    
    \draw[fill] (1) circle [radius=0.1];
    \node at (0.5, 3) {$1$};
    \draw[fill] (2) circle [radius=0.1];
    \node at (3/1.71 + 0.5, 0) {2};
    \draw[fill] (3) circle [radius=0.1];
    \node at (-3/1.71-0.5, 0) {3};
    \end{tikzpicture}
\end{center}
with regular $2$-fold covering graph $\tilde{\Gamma}$:
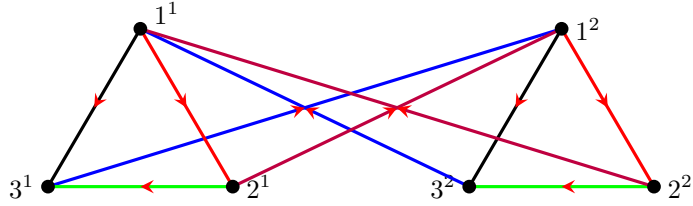
    \begin{center}
    \begin{tikzpicture}[scale = 0.7, line width=1.2]
    
    \coordinate (4) at (-4, 3);
    
    \coordinate (5) at (3/1.71 - 4, 0);
    \coordinate (6) at (-3/1.71 - 4, 0);
    
    \coordinate (7) at (4, 3);
    
    \coordinate (8) at (3/1.71 + 4, 0);
    \coordinate (9) at (-3/1.71 + 4, 0);

    \node at (-4 + 0.5, 3+0.3) {$1^1$};
    
    \node at (3/1.71 - 4 + 0.5, 0) {$2^1$};
    \node at (-3/1.71 - 4 -0.5, 0) {$3^1$};
    
    \node at (4 + 0.5, 3) {$1^2$};
    
    \node at (3/1.71 + 4 + 0.5, 0) {$2^2$};
    \node at (-3/1.71 + 4 - 0.5, 0) {$3^2$};
    
    \path[draw = black, postaction = {on each segment = {mid arrow = red}}]
 
    (4) -- (6);
    \path[draw = black, postaction = {on each segment = {mid arrow = red}}]
    (7) -- (9);
    \path[draw = blue, postaction = {on each segment = {mid arrow = red}}]
    (9) -- (4);
    \path[draw = blue, postaction = {on each segment = {mid arrow = red}}]
    (6) -- (7);
    \path[draw = green, postaction = {on each segment = {mid arrow = red}}]
    (8) -- (9);
    \path[draw = green, postaction = {on each segment = {mid arrow = red}}]
    (5) -- (6);
    \path[draw = red, postaction = {on each segment = {mid arrow = red}}]
    (7) -- (8);
    \path[draw = red, postaction = {on each segment = {mid arrow = red}}]
    (4) -- (5);
    \path[draw = purple, postaction = {on each segment = {mid arrow = red}}]
    (5) -- (7);
     \path[draw = purple, postaction = {on each segment = {mid arrow = red}}]
    (8) -- (4);

    \draw[fill] (4) circle [radius=0.1];
    \draw[fill] (5) circle [radius=0.1];
    \draw[fill] (6) circle [radius=0.1];
    \draw[fill] (7) circle [radius=0.1];
    \draw[fill] (8) circle [radius=0.1];
    \draw[fill] (9) circle [radius=0.1];
    \end{tikzpicture}
    \end{center}
    
Then the following is an arborescence of $\tilde{\Gamma}$ rooted at $3^1$:
   
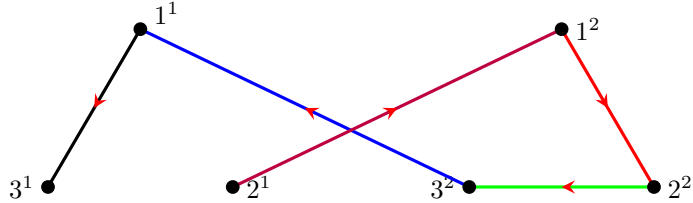
\begin{center}
    \begin{tikzpicture}[scale = 0.7, line width=1.2]
    
    \coordinate (4) at (-4, 3);
    
    \coordinate (5) at (3/1.71 - 4, 0);
    \coordinate (6) at (-3/1.71 - 4, 0);
    
    \coordinate (7) at (4, 3);
    
    \coordinate (8) at (3/1.71 + 4, 0);
    \coordinate (9) at (-3/1.71 + 4, 0);

    \node at (-4 + 0.5, 3+0.3) {$1^1$};
    
    \node at (3/1.71 - 4 + 0.5, 0) {$2^1$};
    \node at (-3/1.71 - 4 -0.5, 0) {$3^1$};
    
    \node at (4 + 0.5, 3) {$1^2$};
    
    \node at (3/1.71 + 4 + 0.5, 0) {$2^2$};
    \node at (-3/1.71 + 4 - 0.5, 0) {$3^2$};
    
    \path[draw = black, postaction = {on each segment = {mid arrow = red}}]
 
    (4) -- (6);
    \path[draw = blue, postaction = {on each segment = {mid arrow = red}}]
    (9) -- (4);
    \path[draw = green, postaction = {on each segment = {mid arrow = red}}]
    (8) -- (9);
    \path[draw = red, postaction = {on each segment = {mid arrow = red}}]
    (7) -- (8);
    \path[draw = purple, postaction = {on each segment = {mid arrow = red}}]
    (5) -- (7);

    \draw[fill] (4) circle [radius=0.1];
    \draw[fill] (5) circle [radius=0.1];
    \draw[fill] (6) circle [radius=0.1];
    \draw[fill] (7) circle [radius=0.1];
    \draw[fill] (8) circle [radius=0.1];
    \draw[fill] (9) circle [radius=0.1];
    \end{tikzpicture}
    \end{center}
However, this arborescence cannot be constructed from a lift of an arborescence of $\Gamma$ rooted at vertex $3$. Any such arborescence must necessarily contain edges $(3, 1)$ and $(1, 3)$, which form a $2$-cycle. This counterexample uses a regular covering graph, and the base graph has no loops or multiple edges. Therefore non-regularity, loops, and multiple edges are not essential impediments to the construction. This construction might be salvageable, but any combinatorial bijection would need to be more complicated than one that involves merely lifting the base arborescence.

Rather than fixing the data of a single arborescence on $\Gamma$ and enumerating over the data of certain sets of $k - 1$ vector fields on $\Gamma$, one might consider applying the construction in the opposite way: fix a set of $k - 1$ vector fields and enumerate over arborescences. However, it is even less clear how this construction would proceed. Moreover, Theorem \ref{bigboi} suggests that the first method is likelier to succeed. Assuming that Conjecture \ref{posconjecture} is true (the arborescence ratio has positive coefficients), the theorem implies that every arborescence of $\Gamma$ is associated to multiple arborescences of $\tilde{\Gamma}$, which suggests that our first construction attempt is closer to the truth.
It is interesting and surprising that Theorem \ref{bigboi} is true despite the lack of an obvious combinatorial relationship between arborescences of $\Gamma$ and of $\tilde{\Gamma}$.

\subsection{The reduced group algebra}
We wish to define a matrix similar to the Laplacian matrix that tracks all the relevant information in an $G$-voltage graph. In order to do so in general, we need to extend our field of coefficients in order to codify the data given by the voltage function $\nu$. Following the language of Reiner and Tseng in 
\cite{reduced}:

\begin{definition}
The \emph{reduced group algebra} of a finite group $G$ over a ring $R$ is the quotient
\begin{align*}
    \conj{R[G]} = \frac{R[G]}{\left\langle\sum_{g \in G} g\right\rangle},
\end{align*}
where $R[G]$ is the group algebra of $G$ over $R$. That is, we quotient the group algebra by the all-ones vector with respect to the basis given by $G$.
\end{definition}

For simplicity, in the remainder of this paper we take $R = \BZ$. Note that if $G \cong \BZ/2\BZ$, then $\conj{\BZ[G]} \cong \BZ$, with the non-identity element of $G$ identified with $-1$. 

Similarly, if $G \cong \BZ/p\BZ$, with $p$ prime, then $\conj{\BZ[G]} \cong \BZ(\zeta_p)$, where $\zeta_p$ is a primitive $p$-th root of unity and the generator $g$ of $G$ is identified with $\zeta_p$. (To see this, note that both rings arise by adjoining to $\BZ$ an element with minimal polynomial $\sum_{i = 0}^{p-1} x^i$.) The fact that the reduced group algebra of prime cyclic $G$ lies in a field extension over $\BQ \supseteq \BZ$ will be important later in giving us nice formulas for the ratio of arborescences described in the introduction.

\subsection{The voltage Laplacian matrix} \label{voltage-laplacian-section}
We now define a generalization of the Laplacian matrix that takes into account voltages:
\begin{definition}
The \emph{voltage adjacency matrix} $\mathscr{A}(G)$ has entries given by
\begin{align*}
    a_{ij} = \sum_{e = (v_i, v_j) \in E} \nu(e) \text{wt}(e),
\end{align*} where we consider $\nu(e)$ as an element of the reduced group algebra $\conj{\BZ[G]}$. That is, the $i,j$-th entry consists of sum of the \emph{volted} weights of all edges going from the $i$-th vertex to the $j$-th vertex. The \emph{voltage Laplacian matrix} $\mathscr{L}(\Gamma)$ is defined as
\begin{align*}
    \mathscr{L}(\Gamma) = D(\Gamma) - \mathscr{A}(\Gamma)
\end{align*}
where $D(\Gamma)$ is the (unvolted) weighted degree matrix as described in Definition \ref{Laplacian}.
\end{definition}

Note that when every edge has trivial voltage, then $\mathscr{L}(\Gamma) = L(\Gamma)$, so that the voltage Laplacian is indeed a generalization of the Laplacian. 
Since we consider the edge weights of $\Gamma$ as indeterminates, we treat the entries of $\mathscr{L}(G)$ as elements of $\conj{\BZ[G]}[\wt(E)]$---that is, the polynomial ring of edge weights with coefficients in the reduced group algebra. We will often abuse notation and refer to this ring as simply $\conj{\BZ[G]}[E]$.

\begin{example}\label{voltlapex}

Let $\Gamma$ be the $\BZ/3\BZ$-voltage graph in Figure~\ref{fig:voltage-graph}. Under the identification $\conj{\BZ[\BZ/3\BZ]} \cong \BZ(\zeta_3)$, the voltage Laplacian of $\Gamma$ is
\begin{align*}
    \mathscr{L}(\Gamma) &= \begin{bmatrix}
        a + b & 0 & 0 \\
        0 & c & 0\\
        0 & 0 & d + e
    \end{bmatrix} - \begin{bmatrix}
        \zeta_3 a & b & 0 \\
        0 & 0 & \zeta_3^2c\\
        \zeta_3^2 d & e  & 0
    \end{bmatrix}\\
    &= \begin{bmatrix}
        (1 - \zeta_3)a + b & -b & \phantom{00000}0\phantom{00000}\\
        0 &c & -\zeta_3^2c\\
        -\zeta_3^2 d & -e & d  + e
    \end{bmatrix}
\end{align*}
\end{example}

\subsection{Notation}\label{sec:notations}
Before reintroducing our main result, we summarize the conventions and notation that will be used consistently throughout the rest of the paper.
For a graph $\Gamma=(V,E,\wt)$ and a covering graph $\tilde\Gamma=(\tilde V,\tilde E,\wt)$,
\begin{itemize}[nosep]
    \item The parameter $n$ refers to the order (number of vertices) of $\Gamma$, and we write $V=\{v_i:i\in[n]\}$
    \item The parameter $k$ refers to the degree of the cover, i.e. $\tilde \Gamma$ is a $k$-fold covering graph of $\Gamma$. When the cover is regular, we have $k=|G|$.
    \item We write the vertices of $\tilde\Gamma$ as $\tilde V=\{v_i^j:i\in[n],j\in[k]\}$, where $v_i^1,\cdots,v_i^k$ are the lifts of $v_i$. For regular covers, we write $\tilde{V}=\{v_i^g:i\in[n],g\in G\}$, associating each vertex with an element $G$ for the purposes of the following bullet point.
    \item In the case of a regular $G$-cover, the edges in $\tilde \Gamma$ are $\tilde E=\{(v_i^h,v_j^{\nu(v_i)h}):(v_i,v_j)\in E,h\in G\}$. When the cover is not necessarily regular, we assign a permutation in $\mathfrak S_k$ to each edge of $\Gamma$. Then an edge $e=(v_a,v_b)$ with permutation $\sigma$ is lifted to the edges $\{(v_a^i,v_b^{\sigma(i)}):i\in[k]\}$; by abuse of notation, we write $\nu(e)=\sigma$. Note the difference between $\sigma(i)$ and group multiplication. For example, in Example~\ref{ex:perm}, the permutation associated to the edge $e$ is $132$ {in one-line notation}.
\end{itemize}

\section{Proof of the Main Theorem} \label{sec: main-thm}

\subsection{Restriction of scalars}
\begin{definition}{
	Let $R$ be a commutative ring, and let $S$ be a free $R$-algebra of finite rank. Let $T$ be an $S$-linear transformation on a free $S$-module $M$ of finite rank. Then we may also consider $M$ as a free $R$-module of finite rank, and $T$ as an $R$-linear transformation; this is known as \emph{restriction of scalars}. We write $\det_R T$ to denote the determinant of $T$ as an $R$-linear transformation. 
}
\end{definition}

Recall that the voltage Laplacian $\mathscr{L}(\Gamma)$ has entries in the reduced group algebra augmented by edge weights: $S = \conj{\BZ[G]}[E]$. Letting $R = \BZ[E]$, we may also consider $\mathscr{L}(\Gamma)$ as an $R$-linear transformation on a $R$-module of rank $(k - 1)n$. Note that due to the definition of the Laplacian (the fact that \textbf{row} entries sum to 0), we will consider our linear transformation to act on the right. 

\begin{example}\label{restrictionexample}
    Returning to Example \ref{voltlapex}, the voltage Laplacian $\mathscr{L}(\Gamma)$ is a matrix that represents a linear transformation on a $\BZ(\zeta_3)[E]$-module with basis vectors indexed by the three vertices of $\Gamma$:
    $$
        \mathscr{L}(\Gamma) = \begin{bmatrix}
        (1 - \zeta_3)a + b & -b & 0\\
        0 &c & -\zeta_3^2c\\
        -\zeta_3^2 d & -e & d  + e
    \end{bmatrix}
    $$
    We may consider this same module as a $\BZ[E]$-module instead, simply by disallowing scalar multiplication outside of the subring $\BZ[E] \subseteq \BZ(\zeta_3)[E]$. Now we look at the basis vectors of the $\BZ[E]$-module. Since the $\BZ[E]$-span of a set of vectors is smaller than its $\BZ(\zeta_3)[E]$-span, however, we will need more basis vectors than before in order to span the entire module. One basis for this module has basis vectors doubly indexed by vertices and the two non-identity group elements of $\BZ/3\BZ$, which shows that the module has $\BZ[E]$-rank $6$. Ordering basis vectors as $v_{1}^{g}, v_{2}^{g}, v_{3}^{g}, v_{1}^{g^2}, v_{2}^{g^2}, v_{3}^{g^2}$, the voltage Laplacian may considered as a $\BZ[E]$-linear transformation, with matrix
    \begin{align*}
        [\mathscr{L}(\Gamma)]_{\BZ[E]} = \begin{bmatrix}
            a + b & -b & 0 & -a & 0 & 0 \\
            0 & c & c &  0 & 0 & c\\
            d & -e & d + e & d & 0 & 0\\
            a & 0 & 0 & 2a+b & -b & 0\\
            0& 0 & -c & 0& c & 0\\
            -d & 0 & 0 & 0 & -e & d+e
        \end{bmatrix}
    \end{align*}
    and the $\BZ[E]$-determinant of this transformation is
    \begin{align*}
        \det_{\BZ[E]} [\mathscr{L}(\Gamma)]&:= \det [\mathscr{L}(\Gamma)]_{\BZ[E]}\\ &= 3a^2c^2d^2 + 3b^2c^2d^2 + 6abc^2d^2 + 9a^2c^2e^2 + 3b^2c^2e^2 + 9abc^2e^2 + 9a^2c^2de + 3b^2c^2de + 12abc^2de
    \end{align*}
\end{example}

\subsection{The prime cyclic case}\label{theprimecycliccase}

The framework provided by restriction of scalar now allows us to prove that Corollary \ref{primecyclic} follows from Theorem \ref{bigboi}.

\begin{proof}[Proof of Corollary \ref{primecyclic} given Theorem \ref{bigboi}]
	The corollary follows from the theorem if we can show that $\det_{\BZ[E]} [\mathscr{L}(\Gamma)] = N_{\BQ(\zeta_p)/\BQ}\left(\det  \left[\mathscr{L}(\Gamma)\right]\right)$. Theorem 1 of \cite{sylvester} states that if $A$ is a commutative ring, if $B$ is a commutative subring of $\mathrm{Mat}_n(A)$, and if $M \in \mathrm{Mat}_m(B)$, then 
	\begin{align*}
	    \det_A [M] = \det_A (\det_B [M])
	\end{align*}
	In this case, let $A := \BQ[E]$. The reduced group algebra $B := \BZ(\zeta_p)[E]$
	may be realized as a subring of $\mathrm{Mat}_{p-1}(A)$, with an element $\alpha$ of $\BZ(\zeta_p)[E]$ being identified with the $\BZ[E]$-matrix corresponding to multiplication by $\alpha$ in $\BZ(\zeta_p)[E]$, where we view $\BZ(\zeta_p)[E]$ as an $\BZ[E]$-module. Note that $\BZ[E]$ and $\BZ(\zeta_p)[E]$ are both commutative. Finally, we let $M = \mathscr{L}(\Gamma)$. But the field norm $N_{\BQ(\zeta_p)/\BQ}(\alpha)$ is defined as the determinant of the linear map $x \mapsto \alpha x$ as a $\BQ$-linear transformation, or, equivalently in our case, a $\BZ$-linear transformation. When extended to $\BQ(\zeta_p)[E]$, this definition shows that
	\begin{align*}
	    \det_{\BZ[E]} \left(\det_{\BZ(\zeta_p)[E]} [\mathscr{L}(\Gamma)]\right) = N_{\BQ(\zeta_p)/\BQ}\left(\det_{\BQ(\zeta_p)[E]}[\mathscr{L}(\Gamma)] \right)
	\end{align*}
	as desired.
\end{proof}

\begin{example}
Let $\Gamma$ be the graph from Figure~\ref{fig:voltage-graph}. We compute $\det [\mathscr{L}(\Gamma)]$ in Example \ref{NVFexample}:
\begin{align*}
    \det [\mathscr{L}(\Gamma)] = (1-\zeta_3) bcd + (1 - \zeta_3)acd + (1-\zeta_3^2)bce + (1- \zeta_3)(1 - \zeta_3^2)ace
\end{align*}
Since voltage is given by $\BZ/3\BZ$, the reduced group algebra is $\BZ(\zeta_3)[E] \subset \BQ(\zeta_3)[E]$, which we treat as an extension over $\BQ$. The Galois norm in this case is the same as the complex norm, since the Galois conjugate of an element of $\BQ(\zeta_3)[E]$ is its complex conjugate. This norm is
\begin{align*}
    \det [\mathscr{L}(\Gamma)] \det [\conj{\mathscr{L}(\Gamma)}]  &= \left((1-\zeta_3) bcd + (1 - \zeta_3)acd + (1-\zeta_3^2)bce + (1- \zeta_3)(1 - \zeta_3^2)ace \right)\\
    &\cdot \left((1-\zeta_3^2) bcd + (1 - \zeta_3^2)acd + (1-\zeta_3)bce + (1- \zeta_3^2)(1 - \zeta_3)ace\right)\\
    &= 3a^2c^2d^2 + 3b^2c^2d^2 + 6abc^2d^2 + 9a^2c^2e^2 + 3b^2c^2e^2 + 9abc^2e^2 + 9a^2c^2de + 3b^2c^2de + 12abc^2de
\end{align*}
which matches $\det [\mathscr{L}(\Gamma)]_{\BZ[E]}$ from Example \ref{restrictionexample}.
\end{example}

\subsection{Triangularization}

In this section we build up the machinery to prove our main result, Theorem \ref{bigboi}.

For ease of notation, we will first fix an ordering on the basis of $\tilde\Gamma$.
Most of the time we won't assume the covering graph to be regular, and using the notations in Section~\ref{sec:notations}, we order the basis vectors of $\tilde\Gamma$ in a colexicographic order, i.e. $v_1^1<\cdots<v_n^1<v_1^2<\cdots<v_n^2<\cdots<v_1^k<\cdots<v_n^k$.

\begin{definition}\label{lowerright}
Let $\{v_1,\cdots,v_n\}$ be the set of vertices of our graph $\Gamma$, let $\tilde \Gamma$ be a $k$-fold cover of $\Gamma$, where vertex $v_i$ is lifted to $v_i^1,\cdots,v_i^k$.
Define $n(k-1)\times n(k-1)$ matrices $D$ and $A$ with basis $v_1^2,\cdots,v_n^2,v_1^3,\cdots,v_n^3,\cdots,v_1^k,\cdots,v_n^k$ as follows.
\[A[v_i^t,v_j^r]=\sum_{e=(v_i^t,v_j^r)}\wt(e)-\sum_{e=(v_j^r,v_i^1)}\wt(e)\]
\[D[v_i^t,v_i^t]=\sum_{e\in E_s(v_i^t)}\wt(e)\]
for $1<t,r\leq k.$ Finally, we define
\[ [\mathscr{L}(\Gamma)]_{\BZ[E]}:=D-A.\]
\end{definition}

In the case of regular covers, we can choose an ordering $g_1,\ldots,g_k$ on the group elements, and use the order  $v_1^{g_1}<\cdots<v_n^{g_1}<v_1^{g_2}<\cdots<v_n^{g_2}<\cdots<v_1^{g_k}<\cdots<v_n^{g_k}$ on the basis vectors of $\tilde{\Gamma}$. If $g_1$ is the identity element, Definition \ref{lowerright} gives the $\BZ$-linearization of the voltage Laplacian. 
Note that in the case of non-regular covers, the matrix cannot be interpreted as the $\BZ$-linearization of a voltage Laplacian. 
 
With this definition in hand, we remind the reader of Theorem ~\ref{bigboi}:

\setcounter{section}{1}
\setcounter{restate}{2}
\begin{restate}
	 Let $\Gamma = (V, E, \wt)$ be an edge-weighted  multigraph, and let $\tilde{\Gamma}$ be a $k$-fold covering graph of $\Gamma$. Then for any vertex $v$ of $\Gamma$ and any lift $\tilde{v}$ of $v$, we have
	\begin{align*}
	\frac{A_{\tilde{v}}(\tilde{\Gamma})}{A_v(\Gamma)} = \frac{1}{k}{\det } [\mathscr{L}(\Gamma)]_{\BZ[E]}
	\end{align*}
	with $[\mathscr{L}(\Gamma)]_{\BZ[E]}$ given by Definition \ref{lowerright}.
\end{restate}
\setcounter{section}{3}

To prove this theorem, we carefully apply a change of basis to the Laplacian matrix of $\tilde{\Gamma}$.  

\begin{lemma}[Triangularization Lemma]\label{trilem}
    Let $\Gamma$ be as in Theorem \ref{bigboi}. Write $L(\tilde{\Gamma})$ with basis vectors ordered as above. Let
    \begin{align*}
        S = \begin{bmatrix}
           \id_{n} & \id_{n} &\dots  & \id_{n}\\
           0_{n}& & & \\
           \vdots  & & \id_{(k - 1)n} & \\
           0_{n} & & &
        \end{bmatrix}
    \end{align*}
    Then the change of basis given by $S$ yields the following block triangularization of $L(\tilde{\Gamma})$:
    \begin{align}\label{triangle}
        S L(\tilde{\Gamma}) S^{-1} = \begin{bmatrix}
             L(\Gamma) & 0\\
             * & \left[\mathscr{L}(\Gamma)\right]_{\BZ[E]}
        \end{bmatrix}
    \end{align}
\end{lemma}
\begin{proof}
    Let $\beta_i = \sum_{j \in [k]} v_{i}^{j}$. Conjugation by $S$ corresponds to a change of basis that maps $v_{i}^{1} \mapsto \beta_i$ and $v_{i}^{j} \mapsto v_{i}^{j}$ when $j \neq 1$. Therefore, all we need to do is examine the action of the linear transformation corresponding to the matrix $L(\tilde{\Gamma})$ on this new basis. Denote this linear transformation as $T$.

    First, we show that
    \begin{align*}
        T(\beta_i) = \sum_{j = 1}^{n} \ell_{ij} \beta_j
    \end{align*}
    where $\ell_{ij}$ is the $(i,j)$-entry of $L(\Gamma)$. To see this, consider the $k$ rows of $L(\tilde{\Gamma})$ corresponding to the fiber $\{v_i^r\}_{r \in [k]}$ of $v_i$. The sum of these rows is equal to $T(\beta_i)$, expressed as a row vector with respect to the standard basis. Choose a column of $L(\tilde{\Gamma})$ corresponding to $v_j^r$.  The entries of this column in the previously mentioned rows correspond (with negative sign) to edges to $v_j^r$ from some element in the fiber of $v_i$.  Since there should be one of these for each edge $(v_i,v_j)$ in $\Gamma$, the sum of these values is $\ell_{ij}$.  Now choose the column of $L(\tilde{\Gamma})$ corresponding to $v_i^r$.  The entries of this column in the previously mentioned rows on the off-diagonal correspond (with negative sign) to edges to $v_i^r$ from some element in the fiber of $v_i$ other than $v_i^r$.  The diagonal entry of the column corresponds (with negative sign) to edges from $v_i^r$ to itself as well as (with positive sign) all edges out of $v_i^r$.  Adding these values, we get $\ell_{ii}$, since there is one edge into $v_i^r$ from a vertex in the fiber of $v_i$ for each edge from $v_i$ to itself and one edge out of $v_i^r$ for each edge out of $v_i$.
    Thus, the sum of the $n$ rows of $L(\tilde{\Gamma})$ corresponding to the fiber $\{v_i^r\}_{r \in [k]}$ is precisely 
    \begin{align*}
    T(\beta_{i}) = \sum_{j  = 1}^{n} \sum_{r \in [k]} \ell_{ij} v_j^r = \sum_{i  =1}^{n} \ell_{ij} \beta_j
    \end{align*}
    as desired. Therefore, the upper-left and upper-right blocks of (\ref{triangle}) are correct.

    The effect of our change of basis on the lower-right $(k-1)n \times (k-1)n$  block of $L(\tilde{\Gamma})$ is to subtract the $i$-th column of $L(\tilde{\Gamma})$, for $i \in [n]$, from columns $i + n, i + 2n, \dots, i + (k - 1)n$. This exactly follows the construction of $[\mathscr{L}(\Gamma)]_{\BZ[E]}$ from Definition \ref{lowerright}. 
    
\end{proof}

\begin{example}
    We illustrate Lemma \ref{trilem} using Example \ref{ex1}, taking $\Gamma$ to be the graph in Figure \ref{fig:voltage-graph} and $\tilde{\Gamma}$ to be its corresponding derived graph as in Figure \ref{fig:derived-graph}. The Laplacian matrix of $\tilde{\Gamma}$ is 
    \begin{align*}
        L(\tilde{\Gamma}) = \begin{bmatrix}
        a + b & -b & 0 & -a & 0 & 0 & 0 & 0 & 0 \\
        0 & c & 0 & 0 & 0 & 0 & 0 & 0 & -c \\
        0 & -e & d + e & 0 & 0 & 0 & -d & 0 & 0\\
        0 & 0 & 0 & a +b & -b & 0 & -a & 0 & 0 \\
        0 & 0 & -c & 0 & c & 0 & 0 & 0 & 0 \\
        -d & 0 & 0 & 0 & -e & d + e & 0 & 0 & 0 \\
        -a & 0 & 0 & 0 & 0 & 0 & a + b & -b & 0\\
        0 & 0 & 0 & 0 & 0  -c & 0 & c & 0 \\
        0 & 0 & 0 & -d & 0 & 0 & 0 & -e & d + e
        \end{bmatrix}.
    \end{align*}
    Then letting $S$ be the $9 \times 9$ change of basis matrix defined in the lemma, we have 
    \begin{align*}
        SL(\tilde{\Gamma})S^{-1} = \begin{bmatrix}
 b	& -b	&  0&	  0	& 0	 & 0&	    0	& 0	  &0\\
 0	& c	& -c&	  0	& 0	 & 0	 &   0	& 0	 & 0\\
-d&	-e &	d+e	&  0&	 0	 & 0	&    0&	 0	 & 0\\
 0	& 0	&  0&	a+b	&-b	 & 0&	   -a&	 0	 & 0\\
 0	& 0	 &-c	 & 0	& c	 & c	 &   0	& 0	  &c \\
-d	& 0	 & 0	&  d&	-e &	d+e &	    d&	 0	&  0\\
-a	& 0	&  0	&  a	& 0	&  0&	2a+b&	-b	 & 0\\
 0	& 0	 & 0&	  0	 &0	& -c&	    0&	 c	&  0\\
 0	& 0	 & 0	& -d&	 0	 & 0	 &   0&	-e	&d+e
 \end{bmatrix}.
    \end{align*}
Recall from Example \ref{ex:laplacian} that
\begin{align*}
    L(\Gamma) = \begin{bmatrix}
b & -b & 0 \\
0 & c & -c \\
-d & -e & d+e
\end{bmatrix},
\end{align*}
and recall from Example \ref{restrictionexample} that 
\begin{align*}
    [\mathscr{L}(\Gamma)]_{\BZ[E]} = \begin{bmatrix}
            a + b & -b & 0 & -a & 0 & 0 \\
            0 & c & c &  0 & 0 & c\\
            d & -e & d + e & d & 0 & 0\\
            a & 0 & 0 & 2a+b & -b & 0\\
            0& 0 & -c & 0& c & 0\\
            -d & 0 & 0 & 0 & -e & d+e
        \end{bmatrix}.
\end{align*}

We indeed find that $L(\Gamma)$ appears as the upper-left block of $SL(\tilde{\Gamma})S^{-1}$ and that $[\mathscr{L}(\Gamma)]_{\BZ[E]}$ appears as the lower-right block of $SL(\tilde{\Gamma})S^{-1}$, as predicted by the lemma. Note also the upper-right block of zeros in $SL(\tilde{\Gamma})S^{-1}$.
\end{example}
We would like to make the following statement:
\begin{align} \label{eq:dream}
    \det[SL(\tilde{\Gamma})S^{-1}]_i^i = \det[L(\tilde{\Gamma})]_i^i
\end{align}
for any $i \in [k]$. This would allow us to obtain an expression for $\frac{A_{\tilde{v}}(\tilde{\Gamma})}{A_v(\Gamma)}$ from Lemma \ref{trilem}, since the left-hand side of \ref{eq:dream} is $$\det[L({\Gamma})]_i^i \det[\mathscr{L}(\Gamma)]_{\BZ[E]}= A_{v_i}(\Gamma)\det[\mathscr{L}(\Gamma)]_{\BZ[E]}$$ and the right-hand side is $$\det[L(\tilde{\Gamma})]_i^i = A_{\tilde{v}_i^1}(\tilde{\Gamma}).$$ Unfortunately, Equation \ref{eq:dream} is \emph{not true} as stated. In general, a given minor does not remain invariant under change of basis. However, Equation \ref{eq:dream} turns out to be \emph{nearly true} in that it is only off by a factor of $k$. Our goal is to instead prove the correct statement 
\begin{align*} 
    \det[SL(\tilde{\Gamma})S^{-1}]_i^i = k\cdot \det[L(\tilde{\Gamma})]_i^i.
\end{align*}
To this end, we define $U = SL(\tilde{\Gamma})S^{-1}$. This is the matrix that we will use to connect the two sides of Theorem \ref{bigboi}. Without loss of generality, assume that we want to root our arborescences of $\Gamma$ at vertex $v_1$ and our arborescences of $\tilde{\Gamma}$ at vertex $v_{1}^{1}$. Then the following result is immediate from Lemma \ref{trilem} and Theorem \ref{MTT2}:

\begin{corollary} \label{tricorollary}
\[\det [U_{1}^1] = A_{v_1}(\Gamma) \det [\mathscr{L}(\Gamma)]_{\BZ[E]}.\]
\end{corollary}

To complete the proof of Theorem \ref{bigboi}, we need to show that
\begin{align*}
    \det [U_{1}^1] = k\cdot \det [L_1^1(\tilde{\Gamma})].
\end{align*}

\subsection{The two-step change of basis}
We will show the above equality by factoring $S$ as $QP$ for some matrices $Q$ and $P$ specified in Section~\ref{sec:proof}.  This means conjugation by $S$ is the same as conjugation by $P$ and then conjugation by $Q$.  In this section we will prove two lemmas that will tell us how conjugation by each of these matrices affects minors.

\begin{lemma}\label{lem1}
Let $L$ be the Laplacian matrix of some graph $\Gamma = (V, E, \wt)$. Fix a basis vector $v_i$, and let $P$ be the change of basis matrix that maps $v_{i} \mapsto \sum_{j=1}^n \alpha_jv_j$ with $\alpha_i\neq 0$. That is, $P$ is the identity matrix but with $\alpha_j$ in entry $(i,j)$ for each $j\in J$.  Then
\begin{align*}
    \det [(PLP^{-1})_{i}^i] = \left(\sum_{j=1}^n \frac{\alpha_j}{\alpha_i}\right)A_{v_i}(\Gamma)
\end{align*}
\end{lemma}
\begin{proof}
First note that we need $\alpha_i\neq0$ in order for $P$ to be invertible.  Otherwise, column $i$ would be all 0's.  If $\alpha_j=0$ for all $j\neq i$, then the statement holds trivially, because $(PLP^{-1})_{i}^i=L_{i}^i$ in this case.  So, we can assume there is some $j\neq i$ such that $\alpha_j\neq 0$.

Without loss of generality, let $i=1$ and $\alpha_2\neq0$.  We can see that $P^{-1}$ is the identity matrix with $\frac{1}{\alpha_1}$ in the (1,1) entry and $\frac{-\alpha_i}{\alpha_1}$ in the $(1,i)$ entry.
$L(\tilde{\Gamma})P^{-1}$ differs from $L(\tilde{\Gamma})$ in that the $i$th column of $L(\tilde{\Gamma})P^{-1}$ is the $i$th column of $L(\tilde{\Gamma})$ with $\frac{\alpha_i}{\alpha_1}$ times the first column of $L(\tilde{\Gamma})$ subtracted from it. $PL(\tilde{\Gamma})P^{-1}$ differs from $L(\tilde{\Gamma})P^{-1}$ only in the first row. However, since we are finding the determinant of $PL(\tilde{\Gamma})P^{-1}$ with the first row and column removed, we are only interested in the lower-right hand $(n-1) \times (n - 1)$ submatrix and can ignore this operation.

Notice that $P$ can be factored as $P_nP_{n-1}...P_2$ where $P_2$ is the identity but with $\alpha_1$ in the $(1,1)$ entry and $\alpha_2$ in the $(1,2)$ entry and the rest of the $P_j$'s are the identity but with $\alpha_j$ in the $(1,j)$ entry.  This also gives a factorization for $P^{-1}$.  We will first focus on $P_2LP_2^{-1}$.

We may interpret $(P_2LP_2^{-1})_1^1$ as a submatrix of the Laplacian of a different graph, which we will denote as $\Gamma^{(2)}$.  We construct $\Gamma^{(2)}$ as follows: the vertices of $\Gamma^{(2)}$ are $v_1^{(2)},...,v_{n}^{(2)}$.  If there is an edge $v_r\to v_s$ in $\Gamma$, then there is an edge $v_r^{(2)}\to v_s^{(2)}$ in $\Gamma^{(2)}$, so $\Gamma^{(2)}$ contains $\Gamma$ as a subgraph. For each edge  $e=(v_2,v_1) \in \Gamma$, we add an additional edge $(v_2^{(2)},v_1^{(2)})$ to $\Gamma^{(2)}$ with weight $\frac{\alpha_2}{\alpha_1}\wt(e)$;  we call this an edge of type 1.  Furthermore, for each such $e=(v_i,v_1)\in\Gamma$ where $i\neq 1,2$, we add the edge $(v_i^{(2)}, v_2^{(2)})$ to $\Gamma^{(2)}$ with weight $\frac{-\alpha_2}{\alpha_1}\wt(e)$ and the edge $(v_i^{(2)}, v_1^{(2)})$ with weight $\frac{\alpha_2}{\alpha_1}\wt(e)$.  The first of these edges will be called an edge of type 2 and the second an edge of type 3.

We can see that $L(\Gamma^{(2)})$ is the same as $L(\Gamma)$ except that (aside from the first row, which remains unchanged) $\frac{\alpha_2}{\alpha_1}$ times the first column is subtracted from first column and added to the second column.  $L_{1}^1(\Gamma^{(2)}) = (P_2LP_2^{-1})_1^1$, so $\det [(P_2LP_2^{-1})_1^1]$ counts the arborescences of $\Gamma^{(2)}$ rooted at $v_1^{(2)}$.

We will divide the arborescences of $\Gamma^{(2)}$ into four categories (See Figure~\ref{fig: arb-change-basis}).
\begin{enumerate}[nosep]
\item Arborescences that do not contain any type 1, type 2, or type 3 edges.  The weighted sum of these arborescences is counted by $A_{v_1}(\Gamma)$ because these are exactly the arborescences that use only edges in the subgraph $\Gamma$ of $\Gamma^{(2)}$.
\item Arborescences that contain a type 2 edge paired with arborescences that differ from these by replacing the type 2 edge with a type 3 edge of the same weight with opposite sign.  For every type 2 edge, there is a type 3 edge of the same weight with opposite sign.  This means that for every arborescence that contains a type 2 edge, there is an arborescence that is the same, except instead of the type 2 edge it has a type 3 edge of the same weight with opposite sign.  The weights of these arborescences cancel out, so the weighted sum of all of these arborescences is 0. 
\item Arborescences that contain a type 1 edge not counted in the previous category (2). We claim that in such an arborescence, every edge lies in the subgraph $\Gamma$ except for the unique type $1$ edge. If such an arborescence contained an edge of type $2$, then since vertex $2^{(2)}$ flows directly to the root the edge of type $2$ could be replaced by its corresponding type $3$ edge and still yield a valid arborescence; the same holds if we start with an edge of type $3$. Thus,  arborescences in this category correspond to arborescences in $\Gamma$ where the edge out of 2 goes directly to 1.  So, they contribute $\frac{\alpha_2}{\alpha_1}$ times the weight of such arborescences in $\Gamma$. 
\item Arborescences that contain an edge of type 3 that are not counted in category (2) either.  These are arborescences where removing the edge $e = (v_j^{(2)}, v_1^{(2)})$ of type 3 and replacing it with the corresponding edge $e'=(v_j^{(2)},v_2^{(2)})$ of type 2 does not give an arborescence. This only happens if adding $e'$ would create a cycle, so we conclude that $v_{j}^{(2)}$ lies downstream from $v_2^{(2)}$ in the arborescence flow. This guarantees that there exists only one type $3$ edge---if we have one type $3$ edge out of $v_j^{(2)}$ and another one out of $v_{j'}^{(2)}$, then both of these vertices lie downstream of $v_2^{(2)}$ but both lead directly to $v_1^{(2)}$, which is a contradiction. Furthermore, such an arborescence cannot contain an edge of type $1$---this would immediately contradict vertex $v_{2}^{(2)}$ lying upstream of $v_{j}^{(2)}$---or an edge of type $2$, since type $2$ edges can always be replaced by their corresponding type $3$ edge and still yield a valid arborescence, which would again land us in category (2). 

We conclude that the only ``added" edge in this arborescence is $e$ itself. Therefore, summing over $j\neq1,2$ we see that these arborescences in this category correspond bijectively to arborescences in $\Gamma$ where the edge out of $v_2$ does not go to $v_1$.  This means that in our sum, they contribute $\frac{\alpha_2}{\alpha_1}$ times the weight of such arborescences in $\Gamma$.
\end{enumerate}

The last two categories combine to contribute $\frac{\alpha_2}{\alpha_1} A_{v_1}(\Gamma)$ to the arborescence count $A_{v_1^{(2)}}\Gamma^{(2)}$. Adding the weighted sums of the arborescences in these four categories, we find $$A_{v_1^{(2)}}(\Gamma^{(2)}) = A_{v_1}(\Gamma) + \frac{\alpha_2}{\alpha_1}A_{v_1}(\Gamma)=\left(\frac{\alpha_1+\alpha_2}{\alpha_1}\right)A_{v_1}(\Gamma).$$

From here, we proceed by induction---essentially, all we need to do is prove that continuing to iterate the previous construction over vertices other than $v_2$ continues to work the way we want.  For $3\leq k\leq n$, we will construct $\Gamma^{(k)}$ from $\Gamma^{(k-1)}$ in the same way we constructed $\Gamma^{(2)}$ from $\Gamma$.  However, here the weights on our new edges will have a factor of $\frac{\alpha_k}{\sum_{j=1}^{k-1}\alpha_j}$ rather than $\frac{\alpha_2}{\alpha_1}$.  We will show that $L(\Gamma^{(k)})$ is $L(\Gamma)$ except that (aside from the first row, which remains unchanged) $\frac{\alpha_j}{\alpha_1}$ times the first column of $L(\Gamma)$ is subtracted from first column and added to the $j$th column for $2\leq j\leq k$.  Note that this means that $L_{1}^1(\Gamma^{(k)}) = (P_kP_{k-1}...P_2LP_2^{-1}...P_{k-1}^{-1}P_k^{-1})_1^1$, so $\det [(P_kP_{k-1}...P_2LP_2^{-1}...P_{k-1}^{-1}P_k^{-1})_1^1]$ counts the arborescences of $\Gamma^{(k)}$ rooted at $v_1^{(k)}$.  We will also show that $$A_{v_1^{(k)}}(\Gamma^{(k)})=\left(\sum_{j=1}^{k}\frac{\alpha_j}{\alpha_1}\right)A_{v_1}(\Gamma).$$

We begin with the Laplacian.  We can see that $L(\Gamma^{(k)})$ is $L(\Gamma^{(k-1)})$ except that (aside from the first row, which remains unchanged) $\frac{\alpha_k}{\sum_{j=1}^{k-1}\alpha_j}$ times the first column of $L(\Gamma^{(k-1)})$ is subtracted from first column and added to the $k$th column.  By our inductive hypothesis, $\frac{\alpha_k}{\sum_{j=1}^{k-1}\alpha_j}$ times the first column of $L(\Gamma^{(k-1)})$ is $\left(\frac{\alpha_k}{\sum_{j=1}^{k-1}\alpha_j}\right)\left( \frac{\sum_{j=1}^{k-1}\alpha_j}{\alpha_1}\right)=\frac{\alpha_k}{\alpha_1}$ times the first column of $L(\Gamma)$.  This shows that $L(\Gamma^{(k)})$ is what we want.

Now we turn to the arborescences.  The same method of counting arborescences in $\Gamma^{(2)}$ from arborescences in $\Gamma$ applies for counting arborescences in $\Gamma^{(k)}$ from $\Gamma^{(k-1)}$.  This means $$A_{v_1^{(k)}}(\Gamma^{(k)}) = A_{v_1^{(k-1)}}(\Gamma^{(k-1)}) + \frac{\alpha_k}{\sum_{j=1}^{k-1}\alpha_j}A_{v_1^{(k-1)}}(\Gamma^{(k-1)}),$$ which gives us what we want by the inductive hypothesis.

Thus, we have shown that $$\det [(PLP^{-1})_1^1]=\det [(P_nP_{n-1}...P_2LP_2^{-1}...P_{n-1}^{-1}P_n^{-1})_1^1]=\left(\sum_{j=1}^{n}\frac{\alpha_j}{\alpha_1}\right)A_{v_1}(\Gamma).$$

\begin{figure}
    \centering
        $$
    L(\Gamma) = \begin{bmatrix}
        a + b & - a & 0 & -b\\
        -c & c + d & 0 & -d\\
        -e & -g & e + g & 0\\
        0 & 0 & -f & f
    \end{bmatrix}, \;\;\;\;\;\; P_2 = \begin{bmatrix}
        1 & 1 & 0 & 0 \\
        0 & 1 & 0 & 0\\
        0& 0 & 1 & 0\\
        0 & 0 & 0 & 1
    \end{bmatrix}$$ 
    
 $$   P_2L(\Gamma) P_2^{-1} = \begin{bmatrix}
        a + b - c & -2a - b + 2c + d & 0 & -b - d\\
        -c & 2c + d & 0 & -d\\
        -e & -g + e & g + e & 0\\
        0 & 0 & -f & f
    \end{bmatrix}
$$
    \caption{A Laplacian matrix before and after applying the change of basis $P_2$ with $\alpha_1 = \alpha_2 = 1$. Note that $(P_2LP_2^{-1})_1^1$ matches the corresponding submatrix of the Laplacian of $\Gamma^{(2)}$ (see Figure \ref{fig: arb-change-basis}).}
    \label{fig: mat-change-basis}
\end{figure}
\begin{figure}
\centering

\subfigure{
\begin{tikzpicture}[scale = 0.9, line width=1.2]
\node () at (-0.25,-0.25) {$3$};
\node () at (3.25,-0.25) {$4$};
\node () at (3.25,3.25) {$2$};
\node () at (-0.25,3.25) {$1$};
\node () at (7-0.25,-0.25) {$3$};
\node () at (7+3.25,-0.25) {$4$};
\node () at (7+3.25,3.25) {$2$};
\node () at (7-0.25,3.25) {$1$};
\draw [line width=0.25mm, fill=black] (0,0) circle (1mm);
\draw [line width=0.25mm, fill=black] (0,3) circle (1mm);
\draw [line width=0.25mm, fill=black] (3,0) circle (1mm);
\draw [line width=0.25mm, fill=black] (3,3) circle (1mm);
\path[draw = black, postaction = {on each segment = {mid arrow = red}}]
(0,3) to node[above] {$a$} (3,3)
(0,3) to node[above] {$b$} (3,0)
(3,3) to [bend right] node[above] {$c$} (0,3)
(3,3) to node[right] {$d$} (3,0)
(0,0) to node[left] {$e$} (0,3)
(3,0) to node[below] {$f$} (0,0)
(0,0) to [bend right] node[below] {$g$} (3,3);
\path[draw=black, ->] (4,1.5) to (6,1.5);
\draw [line width=0.25mm, fill=black] (7,0) circle (1mm);
\draw [line width=0.25mm, fill=black] (7,3) circle (1mm);
\draw [line width=0.25mm, fill=black] (10,0) circle (1mm);
\draw [line width=0.25mm, fill=black] (10,3) circle (1mm);
\path[draw = black, postaction = {on each segment = {mid arrow = red}}]
(7,3) to node[above] {$a$} (10,3)
(7,3) to node[above] {$b$} (10,0)
(10,3) to [bend right=70] node[above] {$c$} (7,3)
(10,3) to [bend right] node[above] {$c$} (7,3)
(10,3) to node[right] {$d$} (10,0)
(7,0) to node[left] {$e$} (7,3)
(7,0) to[bend left] node[left] {$e$} (7,3)
(10,0) to node[below] {$f$} (7,0)
(7,0) to [bend right] node[below] {$g$} (10,3)
(7,0) to [bend left] node[below=6] {$-e$} (10,3);
\node at (5,-1) {A graph $\Gamma$ and the corresponding graph $\Gamma^{(2)}$.
};
\end{tikzpicture}
}

\subfigure{
\begin{tikzpicture}[scale = 0.9, line width=1.2]
\node () at (7-0.25,-0.25) {$1$};
\node () at (7+3.25,-0.25) {$2$};
\node () at (7+3.25,3.25) {$4$};
\node () at (7-0.25,3.25) {$3$};
\draw [line width=0.25mm, fill=black] (7,0) circle (1mm);
\draw [line width=0.25mm, fill=black] (7,3) circle (1mm);
\draw [line width=0.25mm, fill=black] (10,0) circle (1mm);
\draw [line width=0.25mm, fill=black] (10,3) circle (1mm);
\path[draw = black, postaction = {on each segment = {mid arrow = red}}]
(10,3) to node[right] {$d$} (10,0)
(7,0) to node[left] {$e$} (7,3)
(10,0) to node[below] {$f$} (7,0);
\node at (8.5,-1) {An arborescence of type 1.};
\end{tikzpicture}
}

\subfigure{
\begin{tikzpicture}[scale = 0.9, line width=1.2]
\node () at (1-0.25,-0.25) {$3$};
\node () at (1+3.25,-0.25) {$4$};
\node () at (1+3.25,3.25) {$2$};
\node () at (1-0.25,3.25) {$1$};
\node () at (7-0.25,-0.25) {$3$};
\node () at (7+3.25,-0.25) {$4$};
\node () at (7+3.25,3.25) {$2$};
\node () at (7-0.25,3.25) {$1$};
\draw [line width=0.25mm, fill=black] (1,0) circle (1mm);
\draw [line width=0.25mm, fill=black] (1,3) circle (1mm);
\draw [line width=0.25mm, fill=black] (4,0) circle (1mm);
\draw [line width=0.25mm, fill=black] (4,3) circle (1mm);
\path[draw = black, postaction = {on each segment = {mid arrow = red}}]
(4,3) to [bend right] node[above] {$c$} (1,3)
(1,0) to [bend left] node[left] {$e$} (1,3)
(4,0) to node[below] {$f$} (1,0);
\draw [line width=0.25mm, fill=black] (7,0) circle (1mm);
\draw [line width=0.25mm, fill=black] (7,3) circle (1mm);
\draw [line width=0.25mm, fill=black] (10,0) circle (1mm);
\draw [line width=0.25mm, fill=black] (10,3) circle (1mm);
\path[draw = black, postaction = {on each segment = {mid arrow = red}}]
(4+6,3) to [bend right] node[above] {$c$} (1+6,3)
(1+6,0) to [bend left] node[left] {$-e$} (4+6,3)
(4+6,0) to node[below] {$f$} (1+6,0);
\node at (5.5,-1) {A pair of arborescences of type 2.};
\end{tikzpicture}
}

\subfigure{
\begin{tikzpicture}[scale = 0.9, line width=1.2]

\node () at (7-0.25,-0.25) {$3$};
\node () at (7+3.25,-0.25) {$4$};
\node () at (7+3.25,3.25) {$2$};
\node () at (7-0.25,3.25) {$1$};
\draw [line width=0.25mm, fill=black] (7,0) circle (1mm);
\draw [line width=0.25mm, fill=black] (7,3) circle (1mm);
\draw [line width=0.25mm, fill=black] (10,0) circle (1mm);
\draw [line width=0.25mm, fill=black] (10,3) circle (1mm);
\path[draw = black, postaction = {on each segment = {mid arrow = red}}]
(10,3) to [bend right = 70] node[above] {$c$} (7,3)
(7,0) to node[left] {$e$} (7,3)
(10,0) to node[below] {$f$} (7,0);
\node at (8.5,-1) {An arborescence of type 3.};
\end{tikzpicture}
}
\qquad \qquad
\subfigure{
\begin{tikzpicture}[scale = 0.9, line width=1.2]
\node () at (7-0.25,-0.25) {$3$};
\node () at (7+3.25,-0.25) {$4$};
\node () at (7+3.25,3.25) {$2$};
\node () at (7-0.25,3.25) {$1$};
\draw [line width=0.25mm, fill=black] (7,0) circle (1mm);
\draw [line width=0.25mm, fill=black] (7,3) circle (1mm);
\draw [line width=0.25mm, fill=black] (10,0) circle (1mm);
\draw [line width=0.25mm, fill=black] (10,3) circle (1mm);
\path[draw = black, postaction = {on each segment = {mid arrow = red}}]
(7,0) to [bend left] node[right] {$e$} (7,3)
(10,3) to node[right] {$d$} (10,0)
(10,0) to node[below] {$f$} (7,0);
\node at (8.5,-1) {An arborescence of type 4.};
\end{tikzpicture}
}
\caption{Types of arborescences for $\Gamma^{(2)}$ with $\alpha_1 = \alpha_2 = 1$.}
\label{fig: arb-change-basis}

\end{figure}
\end{proof}

\newpage

Here is the next lemma we need:

\begin{lemma}\label{lem2}
Let $R$ be a commutative ring and let $M \in \mathrm{Mat}_{n}(R)$. Let $Q \in GL_n(R)$ such that the $i$-th row and column are each the $i$-th unit vector. 
Then
\begin{align*}
   \det[(QMQ^{-1})_i^i] = \det[M_i^i]
\end{align*}
In other words, the minor of $M$ corresponding to removing the $i$-th row and column is invariant under base change by $Q$. 
\end{lemma}
\begin{proof}Without loss of generality $i = 1$. Write 
\begin{align*}
    Q = \begin{pmatrix}
        1 & 0\\
        0 & Q_{1}^1
    \end{pmatrix}, \;\;\;\;M = \begin{pmatrix}
        * & *\\
        * & M_1^1
    \end{pmatrix}
\end{align*}
Thus,
\begin{align*}
    QMQ^{-1} = \begin{pmatrix}
    * & *\\
    * & (Q_1^1)(M_1^1)(Q_1^1)^{-1}
    \end{pmatrix}.
\end{align*}
Then since
\begin{align*}
    \det[(QMQ^{-1})_i^i] &= \det[((Q_1^1)(M_1^1)(Q_1^1)^{-1})] \\
    &= \det[(M_1^1)]\\
    &=: \det[M_1^1],
\end{align*} we conclude that the desired minor is the same as the corresponding minor of $M$.
\end{proof}

\subsection{Proof of Theorem \ref{bigboi}}\label{sec:proof}
\begin{proof}
  Let $P$ be the change of basis that maps $v_{1}^{1} \mapsto \beta_{1} := \sum_{s \in [k]} v_{1}^{s}$, and let $Q$ be the change of basis that maps $v_{i}^{1} \mapsto \sum_{r \in [k]} v_i^r$ for $i > 1$.
  Note that $P$ satisfies the hypotheses for  Lemma \ref{lem1} with $i=1$ and $Q$ satisfies the hypotheses of Lemma \ref{lem2} with $i=1$.
  Letting $S$ be the matrix from Lemma \ref{trilem}, we have $S = QP$. Thus, by Lemmas \ref{lem1} and \ref{lem2},
    \begin{align*}
        \det [U_1^1] &= \det [(QPL(\tilde{\Gamma})P^{-1}Q^{-1})_1^1]\\
        &= \det [(PL(\tilde{\Gamma})P^{-1})_1^1]\\
        &= k A_{v_{1}^1}(\tilde{\Gamma})
    \end{align*}
    However, from Corollary \ref{tricorollary} we know that
    \begin{align*}
        \det [U_1^1] &= A_{v_1}(\Gamma) \det [\mathscr{L}(\Gamma)]_{\BZ[E]}.
    \end{align*}
    Therefore,
    \begin{align*}
        k A_{v_1^1}(\tilde{\Gamma}) = A_{v_1}(\Gamma) \det [\mathscr{L}(\Gamma)]_{\BZ[E]}
    \end{align*}
    as desired. 
\end{proof}

\section{Vector fields and the voltage Laplacian}\label{sec: vec-fields}

In this section, we discuss the connection between the voltage Laplacian and vector fields on voltage graphs, and its implications for positivity in the $2$-fold cover case.

\subsection{Negative Vector Fields}

\begin{definition}
A \emph{vector field} $\gamma$ of a directed graph $\Gamma$ is a subgraph of $\Gamma$ such that every vertex of $\gamma$ has outdegree $1$ in $\Gamma$. As with arborescences, we define the \emph{weight} of a vector field to be the product of its edge weights, that is $$\text{wt}(\gamma) := \prod_{e \in \gamma} \wt(e).$$  Note that $\wt(\gamma)$ is a degree $n$ monomial with respect to the edge weights of $\Gamma$. Write $C(\gamma)$ for the set of cycles in a vector field $\gamma$, of which there is exactly one in each connected component. If $G$ is abelian, and if $c$ is a cycle of $\gamma$ then we define the voltage of $c$ as $\nu(c) := \prod_{e \in c} \nu(e)$; this product is well-defined when $G$ is abelian.
\end{definition}

The determinant of $\mathscr{L}(\Gamma)$ counts vector fields of $\Gamma$ in the following way:
\begin{theorem}[Chaiken] \label{NVF}
	Let $G$ be an abelian group, and let $\Gamma$ be an edge-weighted $G$-voltage graph. Then
	\begin{align*}
		\sum_{\gamma \subseteq \Gamma} \left[\wt(\gamma) \prod_{c \in C(\gamma)} (1 - \nu(c)) \right] = \det [\mathscr{L}(\Gamma)]
	\end{align*}
	where the sum ranges over all vector fields $\gamma$ of $\Gamma$.
\end{theorem}

\begin{example}\label{NVFexample}
    Let $\Gamma$ be the $\BZ/3\BZ$-voltage graph of example \ref{ex1}. There are four distinct vector fields of $\Gamma$ (see Figure \ref{4nvf}).
    \begin{figure}[h]
    \centering
  \subfigure{
  \begin{tikzpicture}
   [line width=1.2, scale = 0.75]
    \coordinate (1) at (0, 3);
    
    \coordinate (2) at (3/1.71, 0);
    \coordinate (3) at (-3/1.71, 0);
    
    \coordinate (4) at (0, 3 + 0.4);
    \path [draw = black, postaction = {on each segment = {mid arrow = red}}]
    
    (1)--(2)
    
    (3)--(1)

    (2) -- (3);
    
    \draw[fill] (1) circle [radius=0.1];
    \node at (0.5, 3) {$1$};
    \draw[fill] (2) circle [radius=0.1];
    \node at (3/1.71 + 0.5, 0) {2};
    \draw[fill] (3) circle [radius=0.1];
    \node at (-3/1.71-0.5, 0) {3};

    \node at (3/1.71/2 + 0.7, 3/2) {$(b, 1)$};
    \node at (-3/1.71/2 - 0.7, 3/2) {$(d, g^2)$};
    \node at (0, -0.85) {$(c, g^2)$};
    \end{tikzpicture}
}
  \subfigure{
  \begin{tikzpicture}
   [line width=1.2, scale = 0.75]
    \coordinate (1) at (0, 3);
    
    \coordinate (2) at (3/1.71, 0);
    \coordinate (3) at (-3/1.71, 0);
    
    \coordinate (4) at (0, 3 + 0.4);
    \draw (1) arc(270:360+270:0.4);
    \path [draw = black, postaction = {on each segment = {mid arrow = red}}]
    
    
    (3)--(1)
    

    (2) -- (3);
    
    \draw[fill] (1) circle [radius=0.1];
    \node at (0.5, 3) {$1$};
    \draw[fill] (2) circle [radius=0.1];
    \node at (3/1.71 + 0.5, 0) {2};
    \draw[fill] (3) circle [radius=0.1];
    \node at (-3/1.71-0.5, 0) {3};

    \node at (1, 3 + 0.4) {$(a, g)$};
    \node at (-3/1.71/2 - 0.7, 3/2) {$(d, g^2)$};
    \node at (0, -0.85) {$(c, g^2)$};
    \end{tikzpicture}
}
  \subfigure{
  \begin{tikzpicture}
   [line width=1.2, scale = 0.75]
    \coordinate (1) at (0, 3);
    
    \coordinate (2) at (3/1.71, 0);
    \coordinate (3) at (-3/1.71, 0);
    
    \coordinate (4) at (0, 3 + 0.4);
    \path [draw = black, postaction = {on each segment = {mid arrow = red}}]
    
    (1)--(2)
    
    

    (2) to [bend left] (3)
    (3) to [bend left] (2);
    
    \draw[fill] (1) circle [radius=0.1];
    \node at (0.5, 3) {$1$};
    \draw[fill] (2) circle [radius=0.1];
    \node at (3/1.71 + 0.5, 0) {2};
    \draw[fill] (3) circle [radius=0.1];
    \node at (-3/1.71-0.5, 0) {3};

    \node at (3/1.71/2 + 0.7, 3/2) {$(b, 1)$};
    \node at (0, 0.85) {$(e, 1)$};
    \node at (0, -0.85) {$(c, g^2)$};
    \end{tikzpicture}
}
  \subfigure{
  \begin{tikzpicture}
   [line width=1.2, scale = 0.75]
    \coordinate (1) at (0, 3);
    
    \coordinate (2) at (3/1.71, 0);
    \coordinate (3) at (-3/1.71, 0);
    
    \coordinate (4) at (0, 3 + 0.4);
    \draw (1) arc(270:360+270:0.4);
    \path [draw = black, postaction = {on each segment = {mid arrow = red}}]
    
    
    

    (2) to [bend left] (3)
    (3) to [bend left] (2);
    
    \draw[fill] (1) circle [radius=0.1];
    \node at (0.5, 3) {$1$};
    \draw[fill] (2) circle [radius=0.1];
    \node at (3/1.71 + 0.5, 0) {2};
    \draw[fill] (3) circle [radius=0.1];
    \node at (-3/1.71-0.5, 0) {3};
    
    \node at (1, 3 + 0.4) {$(a, g)$};
    \node at (0, 0.85) {$(e, 1)$};
    \node at (0, -0.85) {$(c, g^2)$};
    \end{tikzpicture}
}
\caption{The four vector fields of $\Gamma$}
\label{4nvf}
\end{figure}
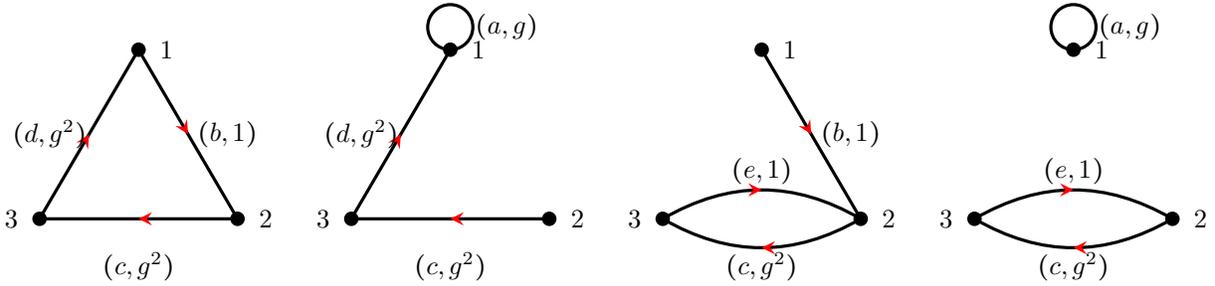
    
    The first three of these vector fields contain one cycle; from left to right, these unique cycles have weights $\zeta_3, \zeta_3,$ and $\zeta_3^2$. The rightmost vector field has two cycles, one with weight $\zeta$ and the other of weight $\zeta^2$. From Example \ref{voltlapex}, we have
    \begin{align*}
        \det [\mathscr{L}(\Gamma)] = (1-\zeta_3) bcd + (1 - \zeta_3)acd + (1-\zeta_3^2)bce + (1- \zeta_3)(1 - \zeta_3^2)ace
    \end{align*}
    The four terms in this expression correspond to the four vector fields of $\Gamma$ as described by the theorem.
\end{example}

We briefly point out the special case $G = \BZ/2\BZ$, which is especially nice because the coefficients in Theorem \ref{NVF} are nonnegative integers.

\begin{definition}
Suppose that $\Gamma$ is a $\BZ/2\BZ$-voltage graph, also called a \emph{signed graph}. A vector field $\gamma$ of $\Gamma$ is a \emph{negative vector field} if every cycle $c$ of $\gamma$ has an odd number of negative edges, so that $\nu(c) = -1$.
\end{definition}

Denote the set of negative vector fields of signed graph $\Gamma$ by $\mathcal{N}(\Gamma)$. Then  Theorem \ref{NVF} may be written as:

\begin{corollary}\label{NVFcor}
    \begin{align*}
        \sum_{\gamma \in \mathcal{N}(\Gamma)} 2^{\#C(\gamma)}\wt(\gamma)  = \det [\mathscr{L}(\Gamma)]
    \end{align*}
\end{corollary}

Corollary \ref{NVFcor} along with Corollary \ref{2foldtheorem} has an immediate further corollary:

\begin{corollary}\label{2coverpositivity}
If $\tilde{\Gamma}$ is a $2-$fold regular cover of $\Gamma$, then the ratio $\frac{A_{\tilde{v}}(\tilde{\Gamma})}{A_v(\Gamma)}$ has positive integer coefficients.
\end{corollary}

Positivity for general covers is still unknown; see Conjecture \ref{posconjecture}.

\subsection{Proofs of Theorem \ref{NVF}}

We now present two proofs of Theorem \ref{NVF}. The first is new, and the second is essentially due to Chaiken.

The first proof proceeds by deletion-contraction, and requires the following lemma.

\begin{lemma}
	Let $\Gamma$ be as in Theorem \ref{NVF} with voltage function $\nu: E \to \conj{\BZ[G]}$,
	let $v$ be any vertex of $\Gamma$, and let $g \in G$. We define a new voltage function $\nu_{v, g}$ given by
	\begin{align*}
		\nu_{v,g}(e) &= \begin{cases}
			g\nu(e) : & \mathrm{if} \; e \in E_{s}(v), e \nin E_t(v)\\
			g^{-1}\nu(e) : &\mathrm{if}\;e \in E_{t}(v), e \nin E_s(v)\\
			\nu(e) : &\mathrm{otherwise}
		\end{cases}
	\end{align*}
	Then: 
	\begin{enumerate}[nosep,label=(\alph*)]
	    \item For any cycle $c$ of $\Gamma$, we have $\nu(c) = \nu_{v,g}(c)$.
	    \item The determinant of the voltage Laplacian of $\Gamma$ with respect to the voltage $\nu$ is equal to the determinant of the voltage Laplacian of $\Gamma$ with respect to $\nu_{v,g}$. That is,
	    \begin{align*}
	        \det [\mathscr{L}(V, E, \wt, \nu)] = \det [\mathscr{L}(V, E, \wt, \nu_{v,g})]
	    \end{align*}
	\end{enumerate}
\end{lemma}
\begin{proof}
\noindent\begin{enumerate}[nosep,label=(\alph*)]
	    \item If $c$ does not contain the vertex $v$, or if $c$ is a loop at $v$, then the voltages of all edges in $c$ remain unchanged. Otherwise, $c$ contains exactly one ingoing edge $e$ of $v$ and one outgoing edge $f$ of $v$, so that 
	    \begin{align*}\nu_{v,h}(c) &= \frac{\nu(c)}{\nu(e)\nu(f)}[g\nu(e)][g^{-1} \nu(f)]\\
	    &= \nu(c)
	    \end{align*}
	    as desired.
	    \item The matrix $\mathscr{L}(V, E, \wt, \nu)$ may be transformed into the matrix $\mathscr{L}(V, E, \wt, \nu_{v,g})$ by multiplying the row corresponding to $v$ by $g$ and multiplying the column corresponding to $v$ by $g^{-1}$, so the determinant remains unchanged. 
	\end{enumerate}
\end{proof}

This lemma will allow us some freedom to change the voltage of $\Gamma$ as needed in the following proof.

\begin{proof}[First proof of Theorem \ref{NVF}]
	Denote the left-hand side of the theorem as
	\begin{align*}
	    \Omega(\Gamma) := \sum_{\gamma \subseteq \Gamma} \left[\wt(\gamma) \prod_{c \in C(\gamma)} (1 - \nu(c)) \right]
	\end{align*}
	
	We proceed by deletion-contraction. The base case is when the only edges of $\Gamma$ are loops. When this happens, $\mathscr{L}(\Gamma)$ is diagonal, with
	\begin{align*}
		\ell_{ii} = \sum_{e = (v_i, v_i) \in E} (1 - \nu(e)) \wt(e).
	\end{align*}
	Thus we have
	\begin{align*}
		\det [\mathscr{L}(\Gamma)] &= \prod_{i = 1}^{n} \left(\sum_{\substack{e = (v_i, v_i) \in E}} \left[1 - \nu(e)\right]\wt(e)\right)
	\end{align*}
	If we expand the product above, each term will correspond to a unique combination of one loop per vertex of $\Gamma$. But such combinations are precisely the vector fields of $\Gamma$, so we obtain
	\begin{align*}
	    \det [\mathscr{L}(\Gamma)] = \Omega(\Gamma)
	\end{align*}
	
	 For the inductive step, assume that there exists at least one edge $e$ between distinct vertices, and assume that the proposition holds for graphs with fewer non-loop edges than $\Gamma$. Using the lemma, we may change the voltage of $\Gamma$ so that $e$ has voltage $1$ without changing either $\Omega(\Gamma)$ or $\det [\mathscr{L}(\Gamma)]$. Without loss of generality, let $v_1 = e_s$ and $v_2 = e_t$. 

	 If $\gamma$ is a vector field of $\Gamma$, then $\gamma$ either contains $e$ or it does not. In the latter case, $\gamma$ is also a vector field of $\Gamma \backslash e$. Clearly all such $\gamma$ arise uniquely from a vector field of $\Gamma \backslash e$. Therefore, there is a weight-preserving bijection between the vector fields of $\Gamma$ not containing $e$ and the vector fields of $\Gamma \backslash e$. 

	 Otherwise, if $e \in \gamma$, then no other edge of the form $(v_1, v_j)$ is in $\gamma$. We define a special type of contraction: let $\Gamma /_{1} e := (\Gamma/e) \backslash E_s(v_1)$. That is, we contract along $e$, and delete all other edges originally in $E_s(v_1)$. Note that the contraction process merges vertices $v_1$ and $v_2$ into a ``supervertex," which we denote $v_{12}$. 
	 
	 Then the vector field $\gamma$ descends uniquely to a vector field $\conj{\gamma}$ on $\Gamma /_{1} e$. Every vector field $\conj{\gamma}$ in $\Gamma /_1 e$ corresponds uniquely to a vector field of $\Gamma$ containing $e$, obtained by letting the unique edge coming out the supervertex $v_{12}$ in $\conj{\gamma}$ be the unique edge coming out of the vertex $v_2$ in $\gamma$, and letting $e$ be the unique edge with source at $v_1$ in $\gamma$. This inverse map shows that the vector fields of $\Gamma$ containing $e$ are in bijection with the vector fields of $\Gamma /_1 e$. This bijection is weight-preserving up to a factor of $\wt(e)$. Finally, note that $\gamma$ and its contraction $\conj{\gamma}$ have the same number of cycles, with the same voltages. If a cycle contains $e$ in $\gamma$, then that cycle is made one edge shorter in $\conj{\gamma}$, but still has positive length since $e$ is assumed to not be a loop. If $c$ is a cycle containing $e$ in $\Gamma$, then because $e$ has voltage $1$, the cycle voltage $\nu(c/e)$ of the contracted version of $c$ is equal to the cycle voltage before contraction. Thus, we may write
	 \begin{align*}
	 	\Omega(\Gamma) = \Omega(\Gamma \backslash e) + \wt(e) \Omega(\Gamma /_1 e)
	 \end{align*}
	 By the inductive hypothesis, since $\Gamma \backslash e$ and $\Gamma /_1 e$ have strictly fewer non-loop edges than $\Gamma$, we have
	 \begin{align*}
	 	\Omega(\Gamma \backslash e) + \wt(e) \Omega(\Gamma /_1  e) = \det [\mathscr{L}(\Gamma \backslash e)] + \wt(e)\det [\mathscr{L}(\Gamma /_1  e)]
	 \end{align*}
	 Note that $\mathscr{L}(\Gamma \backslash e)$ is equal to $\mathscr{L}(\Gamma)$ with $\wt(e)$ deleted from both the $1,1$- and $1,2$-entries. Therefore, via expansion by minors, we obtain
	 \begin{align}\label{eq2}
\det [\mathscr{L}(\Gamma \backslash e)] + \wt(e) \det [\mathscr{L}_{1}^1(\Gamma)] + \wt(e) \det [\mathscr{L}_{1}^2(\Gamma)] = \det [\mathscr{L}(\Gamma)]
	 \end{align}
	 where $\mathscr{L}_{i}^j(\Gamma)$ is the submatrix of $\mathscr{L}(\Gamma)$ obtained by removing the $i$-th row and the $j$-th column. 

	 To construct $\mathscr{L}(\Gamma /_1 e)$ from $\mathscr{L}(\Gamma)$, we disregard the first row of $\mathscr{L}(\Gamma)$, since the special contraction $\Gamma /_1 e$ simply removes the outgoing edges $E_s(v_1)$. Then, we combine the first two columns of $\mathscr{L}(\Gamma)$ by making their sum the first column of $\mathscr{L}(\Gamma /_1 e)$, since when we perform a contraction that merges $v_1$ and $v_2$ into $v_{12}$, we also have $E_t(v_1) \cup E_{t}(v_2) = E_t(v_{12})$. Thus $\mathscr{L}(\Gamma /_1 e)$ is a $(n - 1) \times (n - 1)$ matrix that agrees with both $\mathscr{L}_{1}^1(\Gamma)$ and $\mathscr{L}_{1}^2(\Gamma)$ on its last $n - 2$ columns, and whose first column is the sum of the first columns of $\mathscr{L}_{1}^1(\Gamma)$ and $\mathscr{L}_{1}^2(\Gamma)$. Therefore,
	 \begin{align*}
	 	\det [\mathscr{L}(\Gamma /_1 \; e)] = \det [\mathscr{L}_{1}^1(\Gamma)] + \det [\mathscr{L}_{1}^2(\Gamma)]
	 \end{align*}
	 Substituting into (\ref{eq2}), we obtain
	 \begin{align*}
	 	\det [\mathscr{L}(\Gamma)] &= \det [\mathscr{L}(\Gamma \backslash e)] + \wt(e) \det [\mathscr{L}(\Gamma /_1 \; e)] \\
	 	&= \Omega(\Gamma \backslash e) + \wt(e)\Omega(\Gamma /_1 \; e) \\
	 	&= \Omega(\Gamma)
	 \end{align*}
	 as desired. 
\end{proof}

The second proof of the theorem follows a style similar to Chaiken's proof of the Matrix Tree Theorem in \cite{chaiken}. Chaiken actually proves a more general identity, which he calls the ``All-Minors Matrix Tree Theorem," that gives a combinatorial formula for any minor of the voltage Laplacian. We do not reproduce such generality here, but instead follow a simplified version of his proof, more along the lines of Stanton and White's version of Chaiken's proof of the Matrix Tree Theorem \cite{stanton}. 

\begin{proof}[Second proof of Theorem \ref{NVF} (Chaiken)]
    For simplicity, assume that $\Gamma$ has no multiple edges, since we can always decompose $\det [\mathscr{L}(\Gamma)]$ into a sum of determinants of voltage Laplacians of simple subgraphs of $\Gamma$, which also partitions the sum given in the theorem. We also assume that $\Gamma$ is a complete bidirected graph, since we can ignore edges not in $\Gamma$ by just considering them to have edge weight $0$.
    Write $\mathscr{L}(\Gamma) = (\ell_{ij})$, write $D(\Gamma) = (d_{ij})$, and write $\mathscr{A}(\Gamma) = (a_{ij})$, so that $\ell_{ij} = \delta_{ij}d_{ii} - a_{ij}$. Then the determinant of $\mathscr{L}(\Gamma)$ may be decomposed as
    \begin{align*}
        \det [\mathscr{L}(\Gamma)] = \det[(\delta_{ij}d_{ii} - a_{ij})] &=
        \sum_{S \subseteq [n]} \left[\sum_{\pi \in P(S)} (-1)^{\#C(\pi)} \wt_{\nu}(\pi)\prod_{i \in [n] - S} d_{ii}\right] 
    \end{align*}
    where $P(S)$ denotes the set of permutations of $S$, the set $C(\pi)$ is set of cycles of $\pi$, and $\wt_{\nu}(\pi) := \prod_{i \in S} a_{i, \pi(i)}$. The product of the $d_{ii}$ may be rewritten as a sum over functions $[n] - S \to [n]$, yielding
    \begin{align}
        \nonumber
        \det [\mathscr{L}(\Gamma)] &= \sum_{S \subseteq [n]} \sum_{\pi \in P(S)} (-1)^{c(\pi)} \wt_{\nu}(\pi) \sum_{f: [n] - S \to [n]} \wt(f) \\
        \label{orbit}
        &= \sum_{S \subseteq [n]} \sum_{\pi \in P(S)}\sum_{f: [n] - S \to [n]} (-1)^{c(\pi)} \wt_{\nu}(\pi)  \wt(f) 
    \end{align}
    where $\wt(f)$ denotes the \emph{unvolted} weight of the edge set corresponding to the function $f$, since this part of the product ultimately comes from the degree matrix.
    Thus, the determinant may be expressed as a sum of triples $(S, \pi, f)$ of the above form---that is, we let $S$ be an arbitrary subset of $[n]$, we let $\pi$ be a permutation on $S$, and we let $f$ be a function $[n] - S \mapsto [n]$. 
    
    The permutation $\pi$ can always be decomposed into cycles, and $f$ will sometimes have cycles as well---that is, sometimes we have $f^{(m)}(k) = k$ for some $k \in \BZ$ and $k \in [n] - S$. We can ``swap" cycles between $\pi$ and $f$. Suppose $c$ is a cycle of $f$ that we want to swap into $\pi$. Let the subset of $[n]$ on which $c$ is defined be denoted $W$. Then we may obtain from our old triple a new triple $(S \coprod W, \pi \coprod c , f|_{[n] - S - W})$, where $\pi \coprod c$ denotes the permutation on $S \coprod W$ given by $(\pi \coprod c)(v) = \pi(v)$ if $v \in S$ and $(\pi \coprod c)(v) = c(v)$ if $v \in W$. That is, we ``move" $C$ from $f$ to $\pi$. Similarly, if $c$ is a cycle of $\pi$, then we can obtain a new triple $(S - W, \pi|_{S - W}, f \coprod c)$. Note that these two operations are inverses. 
    
    This process is always weight-preserving: it does not matter whether $c$ is considered as a part of $\pi$ or as a part of $f$, since it will always contribute $\wt(c)$ to the product. However, one iteration of this map will swap the sign of $(-1)^{\#C(\pi)}$, and will also remove or add a factor from $\wt_{\nu}(\pi)$ corresponding to the voltage of $c$. If $\pi$ and $f$ have $k$ cycles between both of them, then there are $2^k$ possibilities for swaps, yielding a free action of $(\BZ/2\BZ)^k$. If we start from the case where $\pi$ is the empty partition, then the sign $(-1)^{\#C(\pi)}$ starts at $1$. Every time we choose to swap a cycle $c$ into $\pi$ from $f$, we flip this sign and multiply by $\nu(c)$, effectively multiplying by $-\nu(c)$. Thus, the sum of terms in (\ref{orbit}) coming from the orbit of the action of $(\BZ/2\BZ)^k$ on $(S, f, \pi)$ is
    \begin{align*}
        \wt(\pi)\wt(f)\prod_{c \in C(\pi) \cup C(f)} (1 - \nu(c))
    \end{align*}
    where $\wt(\pi)$ is now unvolted. This orbit class corresponds to the contribution of one vector field $\gamma$ of $\Gamma$ to the overall sum, where $\gamma$ is the unique vector field such that $\wt(\gamma) = \wt(\pi) \wt(f)$. Thus, summing over all orbit classes, we obtain the desired formula:
    \begin{align*}
        \det [\mathscr{L}(\Gamma)] = \sum_{\gamma \subseteq \Gamma}\left[\wt(\gamma)\prod_{c \in C(\gamma)}(1 - \nu(c))\right]
    \end{align*}
\end{proof}

\section{Conjectures and Future Directions}\label{sec: future}
We end our paper by a discussion of several unanswered questions and possible future research directions.
\subsection{Interpreting the restriction-of-scalars determinant}

In the case where the voltage group $G$ is prime cyclic, Corollary \ref{primecyclic} yields a computationally nice interpretation of Theorem \ref{bigboi}: the $\BZ$-determinant is really a field norm, which may be computed in ways other than restriction of scalars---for example, as a product of Galois conjugates. This result could be extended if there existed an analogue to the field norm for arbitrary reduced group algebras, or indeed for general free algebras of finite rank. A good first step might be to consider abelian groups.

\begin{problem}
    Let $R$ be a commutative ring, and let $A$ be a free algebra over $R$ of finite rank. Let $\alpha \in A$. Find an alternative expression or interpretation of $\det_{R} [\alpha]$, where the multiplicative action of $\alpha$ is viewed as a linear transformation on the $R$-module $A$, analogous to a field norm. Useful special cases include $R = \BZ$ or $\BQ$, when $A$ is commutative, and/or when $A$ is the group algebra or reduced group algebra of some finite group $G$.
\end{problem}
\subsection{Positivity of the ratio and possible combinatorial expression using vector fields}

By Corollary \ref{integrality}, the ratio $\frac{A_{\tilde{v}}(\tilde{\Gamma})}{A_v(\Gamma)}$ is a homogeneous polynomial with integer coefficients. We further conjecture that these coefficients are positive. In more detail:

\setcounter{section}{1}
\setcounter{restate-conj}{6}
\begin{restate-conj}
	 Let $\Gamma$ be a directed graph, $\tilde{\Gamma}$ a $k$-cover of $\Gamma$, $v$ a vertex of $\Gamma$ and $\tilde{v}$ a lift of $v$ in $\tilde{\Gamma}$.  If the edge weights of $\Gamma$ are indeterminates then the polynomial $\frac{A_{\tilde{v}}(\tilde{\Gamma})}{A_v(\Gamma)}$ has positive coefficients.
\end{restate-conj}
\setcounter{section}{5}
Corollary \ref{2coverpositivity} gives positivity for regular $2$-fold cover, and the following proposition gives a way to extend that result to all regular covers by $2$-groups. However, in the case of general regular covers, we do not have a concrete combinatorial interpretation of $\det[\mathscr{L}(\Gamma)]_{\BZ[E]} $. Such an interpretation would probably be the cleanest way to prove Conjecture \ref{posconjecture}.

\begin{proposition}
Suppose we have the exact sequence of groups $1\to N\to G\to H\to 1$, where $N$ and $H$ have the property that every regular $N$- (resp. $H$-) cover satisfies the positivity conjecture. Then every regular $G$-cover satisfies the positivity conjecture.
\end{proposition}

\begin{proof}
Let $\Gamma$ be a graph, let $\Gamma_G$ be a regular $G$-cover of $\Gamma$, and let $\Gamma_H$ be the image of $\Gamma_G$ under the projection map $G\to H$. We will show that both the covers $\Gamma_H\to \Gamma$ and $\Gamma_G$ to $\Gamma_H$ are regular, and therefore that those arborescence ratios are positive. Therefore, the arborescence ratio for the cover $\Gamma_G\to \Gamma$ is the product of these ratios, and therefore positive as well.

Since $G$ is a group extension of $H$ by $N$, we will write elements of $G$ as ordered pairs $(h,n), h\in H, n\in n$, where $n\mapsto (h,n)\mapsto h$ under our exact sequence. Multiplication of these elements involves Schreier theory: \[(h_1,n_1)\cdot (h_2,n_2) = (h_1h_2,f(h_1,h_2,n_1)n_2),\] where $f(h_1,h_2,n_1)$ is an element of $N$ that is independent of $n_2$ (see, for example, \cite{morandi}).

To show that a given cover is regular, we need to demonstrate voltages for the edges of the base graph that give us the desired cover. First, we show that $\Gamma_H$ is regular over $\Gamma$.

Let $e:v\to w$ be an edge in $\Gamma$, with $G$-voltage $g = (h_g,n_g)$. Then $e$ lifts to the set of edges $\{e^{g'}:v^{g'}\to w^{gg'}\}$ in $\Gamma_G$, and these edges project to the set $\{e^h: v^h\to w^{h_gh}\}$ in $\Gamma_H$. Therefore, in the cover $\Gamma_H\to\Gamma$, we can set the voltage of $e$ to be $h_g$, and this gives a regular $H$-cover.

Now we show that $\Gamma_G$ is a regular $N$-cover of $\Gamma_H$. This is more challenging since $H$ is not necessarily a subgroup of $G$. Consider the edge $e$ of $\Gamma$ from above, and let $h\in H$. The edge $e^h: v^h\to w^{h_gh}$ in $\Gamma_H$ is covered by the edges $\{e^{(h,n)} | n\in N\}$ in $\Gamma_G$. Computations in $G$ tell us that \[e^{(h,n)} : v^{(h,n)}\to w^{(h_g,n_g)\cdot (h,n)} = w^{(h_gh,f(h_g,hn_g)n)}.\] Set the voltage on $e_h$ to be $f(h_ghn_g)\in N$. Then \[(e^h)^n: (v^h)^n\to (w^{h_gh})^{f(h_ghn_g)n},\] so identifying $(e^h)^n$ with $e^{(h,n)}$ and likewise for vertices gives us $\Gamma_G$ as a regular cover of $\Gamma_H$.
\end{proof}

\begin{corollary}
Let $G$ be a 2-group. Then every regular $G$-cover satisfies the positivity conjecture.
\end{corollary}

\begin{proof}
By Corollary \ref{2coverpositivity}, this result holds in the case of $\mathbb{Z}/2\mathbb{Z}$. Since $G$ is a 2-group, $\mathbb{Z}/2\mathbb{Z}$ is a normal subgroup, and so the proposition can be applied inductively.
\end{proof}

\begin{problem}
    Find a combinatorial interpretation of the polynomial $\frac{1}{k}\det [\mathscr{L}(\Gamma)]_{\BZ[E]} = \frac{A_{\tilde{v}}(\tilde{\Gamma})}{A_v(\Gamma)}$, assuming Conjecture \ref{posconjecture} is true.
\end{problem}

Vector fields are a potential source of a combinatorial interpretation for the arborescence ratio. We observed that in a $k$-fold cover, the ratio $\frac{A_{\tilde{v}}(\tilde{\Gamma})}{A_v(\Gamma)}$ always appears to be a product of $(k-1)$ weighted sums of vector fields. 
\begin{conjecture}
 Let $\tilde \Gamma$ be a $k$-fold cover of $\Gamma$, then
\[\frac{A_{\tilde v}{(\tilde\Gamma)}}{A_{\tilde v}{(\Gamma)} }=\sum_{(\gamma_1,\cdots,\gamma_{k-1})\in \mathcal V^{k-1}}f(\gamma_1,\cdots,\gamma_{k-1})\prod_{i=1}^{k-1}\wt(\gamma_i)  \] 
where $\mathcal V$ is the set of vector fields of $\Gamma$, and $f$ is an $\mathbb Z_{\geq 0}$-valued function.
\end{conjecture}

Moreover, as we look over all possible $k$-fold cover $\tilde\Gamma$, the ratios exhaust all possible $(k-1)$-tuples of vector fields of the base graph, which is only known in the $2$-fold case. This observation motivates the following conjecture stated in terms of random covers.

\begin{conjecture}\label{random}
    Let $\Gamma=(E,V)$ be a graph, fix a vertex $v$ with non-trivial arborescence. Let $\Gamma'$ be a random $k$-fold cover of $\Gamma$, assuming uniform distribution. Then the expected value of the ratio of arborescence is
    \[\mathbb{E}\left[\frac{\mathcal A_{v'}(\Gamma')}{\mathcal A_{v}(\Gamma)}\right]=\frac{1}{k}\left(\sum_{\gamma\in \mathcal V}\wt (\gamma) \right)^{k-1}=\frac{1}{k}\prod_{w\in V}\left(\sum_{\alpha\in E_s(w)}\wt (\alpha) \right)^{k-1}\]
where $\mathcal V$ is the set of vector fields of $\Gamma$.
\end{conjecture}
Conjecture \ref{random} is an alternative approach to positivity via a `pigeon-hole' like argument: assuming the ratio of some covering graph has a negative coefficient, some cancellation shall happen as we sum over all possible covers; this might cause the expected value to not be `large enough' to match Conjecture \ref{random}.
\section*{Acknowledgements}

This research was carried out as part of the 2019 Combinatorics REU program at the University of Minnesota, Twin Cities, supported by NSF RTG grant DMS-1148634.
We would like to thank Vic Reiner for his expertise in the relevant literature, and for our discussions with him on generalizing the main result. We would also like to thank Pavlo Pylyavskyy for providing us with the original version of conjecture, and for his suggestions on how we might initially proceed.

\bibliographystyle{alpha}
\bibliography{references.bib}

\end{document}